\documentclass[reqno, 11pt]{amsart}

\makeatletter
\@namedef{subjclassname@2020}{%
 \textup{2020} Mathematics Subject Classification}
\makeatother

\usepackage{amsfonts, amsthm, amsmath, amssymb, bm}
\usepackage{hyperref}

\usepackage{graphicx}
\hypersetup{colorlinks=false}
\usepackage{bbm} 
\usepackage{tabu}
\usepackage[margin=1in]{geometry}
\usepackage{helvet}

\RequirePackage{mathrsfs} \let\mathcal\mathscr
 
\usepackage{hyperref}
\usepackage{dsfont}
\usepackage{upgreek}
\usepackage{mathabx,yfonts}
\usepackage{enumerate}

\numberwithin{equation}{section}

\DeclareMathOperator*{\Osum}{\sum{}^*}

\newtheorem{theorem}{Theorem}[section] 
\newtheorem{lemma}[theorem]{Lemma}

\newtheorem{corollary}[theorem]{Corollary}
\newtheorem{conjecture}[theorem]{Conjecture}

\theoremstyle{definition}
\theoremstyle{terminology}

\newtheorem{remark}[theorem]{Remark}

\newtheorem{definition}[theorem]{Definition}
\newtheorem{terminology}[theorem]{Terminology}

\let\ccc\c

\renewcommand{\phi}{\varphi}

\newcommand{\FF}{\mathbb{F}}

\renewcommand{\leq}{\leqslant}

\renewcommand{\geq}{\geqslant}

\renewcommand{\bar}{\overline}

\renewcommand{\c}{\mathbf{c}}

\renewcommand{\b}{\mathbf{b}}

\renewcommand{\r}{\mathbf{r}}

\DeclareMathOperator{\Gal}{Gal}

\DeclareMathOperator{\moo}{mod} 
\renewcommand{\bmod}[1]{\,(\moo{#1})}

\DeclareSymbolFont{bbold}{U}{bbold}{m}{n}
\DeclareSymbolFontAlphabet{\mathbbold}{bbold}

\newcommand{\md}[1]{ \left(\textnormal{mod}\ #1\right)}

\newcommand{\Q}{\mathbb{Q}}
\newcommand{\F}{\mathbb{F}}
\newcommand{\N}{\mathbb{N}}

\newcommand{\Z}{\mathbb{Z}}
\renewcommand{\l}{\left}

\renewcommand{\r}{\right}
\renewcommand{\b}{\mathbf}
\renewcommand{\c}{\mathcal}
\renewcommand{\epsilon}{\varepsilon}

\renewcommand{\leq}{\leqslant}
\renewcommand{\geq}{\geqslant}
\renewcommand{\#}{\sharp}

\title{Elliptic fibrations and $3\cdot 2^k$}
 
\author{P. Koymans}
\address{Mathematisch Instituut \\
Universiteit Utrecht \\
Postbus 80.010, 3508 TA Utrecht \\
Netherlands}
\email{p.h.koymans@uu.nl}

\author{C. Pagano} 
\address{Department of Mathematics \\ 
Concordia University \\ 
Montreal H3G 1M8 \\ 
Canada} 
\email{carlo.pagano@concordia.ca}
 
\author{E. Sofos} 
\address{Department of Mathematics\\
University of Glasgow \\ 
G12~8QQ United Kingdom}
\email{efthymios.sofos@glasgow.ac.uk}
 
\date{}
 
\begin{document} 
\begin{abstract} 
We determine the order of magnitude for all exponential moments of the rank in a broad class of elliptic fibrations and for the $3\cdot 2^k$-torsion in the class group of quadratic fields.
\end{abstract}

\maketitle

\setcounter{tocdepth}{1}
\tableofcontents 

\section{Introduction}
Let $P \in \Z[t_1, \dots, t_n]$ be non-zero and let $r_1, r_2, r_3$ be fixed distinct integers. Consider the elliptic fibration $f: \mathcal{E} \rightarrow \mathbb{A}^n$ given by
$$ 
\mathcal{E}: \quad P(t_1, \dots, t_n) y^2 = (x - r_1) (x - r_2) (x - r_3).
$$ 
We let $E(\mathbf{t})$ be the elliptic curve given by substituting $t_1, \dots, t_n$ and denote its rank by $\mathrm{rk}(E(\mathbf{t}))$.

\begin{theorem}
\label{tFibration}
Let $n \geq 1$, $P \in \Z[t_1, \dots, t_n]$ and $r_i \in \Z$ be as above and let $\kappa > 1$ be arbitrary. Then there exist $c, C > 0$ such that for all sufficiently large $B$ we have 
\[ 
cB^n \leq \sum_{\substack{\mathbf{t} \in \Z^n, \ P(\mathbf{t}) \neq 0 \\ \max_i |t_i| \leq B }} \kappa^{\mathrm{rk}(E(\mathbf{t}))} \leq CB^n.
\] 
\end{theorem} 

For an integer $n \geq 1$, let $h_n(d) := \# \mathrm{Cl}(\Q(\sqrt{d}))[n]$.\begin{theorem}
\label{t12}Fix $k \in \Z_{\geq 1}$ and let $n = 3 \cdot 2^k$. There exist $c', C' > 0$ such that for $X \geq 3$ we have 
\[
c' X \log X \leq \sum_{|d| \leq X} h_n(d) \leq C' X \log X,
\]
where the sum is over integer fundamental discriminants of quadratic fields.
\end{theorem}

\subsection{New ingredients} 
We summarize the new ideas in the case of class groups. Gauss proved that $2h_2(d)$ is essentially a multiplicative function, however, it is well-known that $h_4(d)$ has no obvious multiplicative structure. To estimate the average of $h_{12}(d) = h_3(d) h_4(d)$, the standard approach in the literature~\cite{FK4, HBCongruent} leads to a character sum of the shape
\begin{align}
\label{eNaiveMethod}
\sum_{\b t \in \Z^4\cap \mathcal{D}(B)} \sum_{\substack{\b d \in \Z_{\geq 1}^4 \\ d_0 d_1 d_2 d_3 = F(\b t)}} \mu(d_0 d_1 d_2 d_3)^2 \left(\frac{d_0}{d_1}\right) \left(\frac{d_2}{d_3}\right),
\end{align}
where $(\frac{\cdot}{\cdot})$ is the Jacobi quadratic symbol, $F$ is the discriminant polynomial of the cubic form $t_0 X^3 + t_1 X^2Y + t_2 X Y^2 + t_3 Y^3$ and $\mathcal{D}(B)$ is a fundamental domain for the action of $\mathrm{GL}_2(\Z)$ on binary cubic forms of discriminant bounded by $B$. Unfortunately, the current state of the art cannot handle equidistribution for mutual quadratic symbols between the divisors of a thin integer sequence such as the values of a polynomial $F$.

To deal with this, we majorize $h_4(mn)$ by a function $g(m, n)$ given by the size of the kernel of 
\begin{align}
\label{eRedeiMethod}
\begin{pmatrix}
\ast & (\frac{p_2}{p_1}) & \dots & (\frac{p_r}{p_1}) \\
(\frac{p_1}{p_2}) & \ast & \dots & (\frac{p_r}{p_2}) \\
\vdots & \vdots & \ddots & \vdots \\
(\frac{p_1}{p_r}) &(\frac{p_2}{p_r}) & \dots & \ast
\end{pmatrix},
\end{align}
where $p_i$ are the odd prime divisors of $m$ and the starred entries are a diagonal twist depending on $n$ modulo $m$, see Definition~\ref{def:firstdefn}. The function $g(m, n)$ is periodic in $n$ modulo $m$ and has a weak multiplicative property only after averaging congruence classes. This allows us to introduce sieving ideas of Nair--Tenenbaum \cite{NT} into this problem, see Definition~\ref{dClasses} for the technical set-up. 

The point where sieving and algebra meet can be explained informally as follows: for the minor of the matrix \eqref{eRedeiMethod} consisting of primes $p_1, \ldots, p_k$ with $\prod_{i = 1}^k p_i \leq X^{\epsilon}$ for a small fixed $\epsilon$, we show equidistribution. The contribution of the large primes is controlled via the Nair--Tenenbaum sieve procedure. Once the minor is known to be almost invertible, linear algebra gives a lower bound for the rank of the matrix and thus an upper bound for the size of its kernel.

We now describe how to prove equidistribution of the minor in the simplest case of $h_{12}$. The sieving procedure converts averages over thin sequences into complete averages of the form 
\begin{equation}
\label{eqh412}
\sum_{\substack{\b d \in \Z_{\geq 1}^4 \\ d_0 d_1 d_2 d_3 \leq X}} \mu(d_0 d_1 d_2 d_3)^2 h(d_0 d_1 d_2d_3) \left(\frac{d_0}{d_1}\right) \left(\frac{d_2}{d_3}\right),
\end{equation}
where $h$ is a general non-negative multiplicative function (it is worth comparing the above with \eqref{eNaiveMethod}). In the case $h=1$ and $h=\kappa^{\omega(d)}$ these sums have previously been treated by Fouvry--Kl\"uners respectively in~\cite{FK4} and~\cite{FKWeighted}. We handle the sum \eqref{eqh412} by generalizing their work. Simplifications to their method are introduced, stemming from analytic tools recently appearing in the literature such as the \texttt{LSD} method of Granville--Koukoulopoulos~\cite{GK} and large sieve results for hyperbolic regions \cite{Wil} together with the fact that we only need an upper bound.

Finally, for higher ranks $3 \cdot 2^k$ we use the inequality $h_{2^k} \leq h_2 (h_4 /h_2)^{k - 1}$ to bound the sum over $d$ in Theorem~\ref{t12} by a higher moment of $h_4$ and then we apply our majorizing idea as described above. Well-known analogies between the $2$-Selmer group and $h_4$ allow us to exploit all the ideas above in the context of Theorem~\ref{tFibration} with the caveat that the character sums analogous to~\eqref{eqh412} (first appearing in Heath-Brown \cite{HBCongruent}) are somewhat more involved.

\subsection{Previous results on ranks} 
If we knew Park--Poonen--Voight--Wood's conjecture~\cite{1469.11173} that ranks of elliptic curves over $\Q$ are uniformly bounded, then Theorem~\ref{tFibration} would follow immediately. Our result proves the conjecture `on average' for many thin families of elliptic curves. It was previously only known for linear polynomials by the work of Heath-Brown \cite{HBCongruent}, Kane \cite{Kane} and Smith \cite{Smith2}. When $P$ is an integer polynomial in one variable, Silverman \cite{Silverman1} proved $\mathrm{rank} \ E(t) \geq \mathrm{rank} \ E(\mathbb{A}^1)$ for all but finitely many $t$, where $E(\mathbb{A}^1)$ is the elliptic curve over the function field $\Q(t)$, see also N\'eron \cite{Neron} for a more general but slightly weaker result. Based on these investigations, he made the following conjecture, which constitutes a natural analogue of Goldfeld's well-known conjecture for quadratic twists.

\begin{conjecture}[Silverman, \cite{Silverman2}]For almost all $t \in \Q$ ordered by height we have 
\[ 
\mathrm{rank} \, E(\mathbb{A}^1) \leq \mathrm{rank} \, E(t) \leq 1 + \mathrm{rank} \, E(\mathbb{A}^1).
\] 
\end{conjecture} 

Silverman calls this conjecture \textit{`reasonable yet a difficult question'}. For certain special fibrations, the lower bound and upper bound were achieved for infinitely many fibres by Colliot-Th\'el\`ene--Skorobogatov--Swinnerton-Dyer \cite{CTSSD} conditionally on Shinzel's hypothesis and finiteness of Sha. There are also upper bounds for the average rank in \cite{FP, Michel} that rely on the veracity of BSD and GRH for elliptic curves. As it stands, the upper bound in Silverman's conjecture seems out of reach of current techniques. 

\subsection{Previous results on torsion}
The average of $h_n(d)$ has only been obtained for $n = 4$ by Fouvry--Kl\"uners \cite{FK, FK4, FKWeighted} and $n = 3$ by Davenport--Heilbronn \cite{DH} with second order terms given by Bhargava--Shankar--Tsimerman~\cite{MR3090184} and Taniguchi--Thorne~\cite{MR3127806}. Davenport--Heilbronn's result has been recently extended to the non-abelian setting by Lemke Oliver--Wang--Wood~\cite{woodwangoliver}. The order of magnitude for the average of $h_6(d)$ was determined in \cite{CKPS2}. Finally, the striking methods of Smith \cite{Smith1} allow one to find the average of $h_n(d)$ for $n$ an arbitrary power of $2$.

\subsection{Structure of the paper}
We majorize the rank by a moment of the $4$-class rank (resp. $2$-Selmer rank) in \S\ref{sRedei}. In \S\ref{sNT} we adapt the Nair--Tenenbaum method \cite{NT} to our setting of general majorants. The application of this result will give rise to certain moments weighted by fairly general multiplicative functions; these moments are treated in \S\ref{sMoments} by adapting work of Fouvry--Kl\"uners \cite{FKWeighted} and Heath-Brown \cite{HBCongruent}. In \S\ref{sFinal} we combine the various ingredients from the previous sections to prove Theorem~\ref{tFibration} in \S\ref{ss:prf2} and Theorem~\ref{t12} in \S\ref{ss:prf1}.

\subsection*{Notation} 
We will make use of the following notation throughout the paper.
\begin{itemize}
\item The square-free part of an integer $n \neq 0$ is by definition $n/s$, where $s$ is the largest divisor of $n$ that is a square.
\item If $n$ is an integer, we define $\chi_n: G_\Q \rightarrow \mathbb{F}_2$ to be the quadratic character corresponding to $\Q(\sqrt{n})$. This character is surjective if $n$ is not a square.
\item We write $\Delta(n)$ for the discriminant of $\Q(\sqrt{n})$.
\item If $n$ is an odd integer, then we define $n^\ast$ to be the unique integer such that $|n^\ast| = |n|$ and $n^\ast \equiv 1 \bmod 4$.
\item If $A$ is an abelian group, we write $\mathrm{rk}_{2^n} A := \dim_{\FF_2} 2^{n - 1} (A[2^n])$.
\item We write $P^+(n)$ and $P^-(n)$ respectively for the largest and smallest prime divisor of an integer $n>1$. 
By convention, we set $P^+(1) = 1$ and $P^-(1) = \infty$.
\end{itemize}

\subsection*{Acknowledgements}
The first author gratefully acknowledges the support of Dr. Max R\"ossler, the Walter Haefner Foundation and the ETH Z\"urich Foundation. He also acknowledges the support of the Dutch Research Council (NWO) through the Veni grant ``New methods in arithmetic statistics''.

This paper was initiated when ES visited PK at the Institute for Theoretical Studies in Z\"urich and completed during the authors' visits to the Max Planck Institute for Mathematics in Bonn. We sincerely appreciate the generous hospitality and financial support provided by both institutes.
 
\section{R\'edei majorants}
\label{sRedei}
\subsection{Definition of R\'edei majorants}
We start by giving the following definition: 

\begin{definition}
Fix $A > 1$ and fix a function $g: \{(m, n) \in \N^2 : \gcd(m, n) = 1\} \rightarrow [0 , \infty)$. We assume that $g$ is periodic in its second argument, i.e.
\begin{align}
\label{ePeriod}
g(m, n) = g(m, n + m)
\end{align}
for all coprime $m, n \in \N$. We say that $f: \N \rightarrow [0 , \infty)$ is $(A, g)$-R\'edei majorized if for every $\epsilon > 0$, there exists $C_\epsilon > 0$ such that for all coprime $m, n$ we have 
\begin{align}
\label{efWSub}
f(mn) \leq g(m, n) \min\left(A^{\Omega(n)}, C_\epsilon n^\epsilon\right).
\end{align}
\end{definition}

\subsection{R\'edei matrices}
In this subsection we explain how to calculate the narrow $4$-rank of the class group. Let $m$ be a square-free integer and write $\Delta(m)$ for the corresponding quadratic discriminant. Let $p_1 < \dots < p_r$ be the prime divisors of $\Delta(m)$ ordered by their size. The quadratic character $\chi_m: G_\Q \rightarrow \mathbb{F}_2$, corresponding to the field $\Q(\sqrt{m})$, can be uniquely decomposed as
\[
\chi_m = \sum_{i = 1}^r \rho_i,
\]
where each $\rho_i: G_\Q \rightarrow \mathbb{F}_2$ is a quadratic character with conductor a power of $p_i$. If $p_i \neq 2$ (which certainly holds if $i > 1$), then the conductor equals $p_i$ and we have $\rho_i = \chi_{p_i^\ast}$. If $p_i = 2$, then we have $\rho_i \in \{\chi_{-2}, \chi_{-1}, \chi_2\}$. To $m$, we associate the R\'edei matrix $R(m)$ through
\[
R(m) := 
\begin{pmatrix}
\ast & \rho_2(\mathrm{Frob}_{p_1}) & \rho_3(\mathrm{Frob}_{p_1}) & \dots & \rho_r(\mathrm{Frob}_{p_1}) \\
\rho_1(\mathrm{Frob}_{p_2}) & \ast & \rho_3(\mathrm{Frob}_{p_2}) & \dots & \rho_r(\mathrm{Frob}_{p_2}) \\
\rho_1(\mathrm{Frob}_{p_3}) & \rho_2(\mathrm{Frob}_{p_3}) & \ast & \dots & \rho_r(\mathrm{Frob}_{p_3}) \\
\vdots & \vdots & \vdots & \ddots & \vdots \\
\rho_1(\mathrm{Frob}_{p_r}) & \rho_2(\mathrm{Frob}_{p_r}) & \rho_3(\mathrm{Frob}_{p_r}) & \dots & \ast
\end{pmatrix}
,
\]
where the starred entries are determined by the rule that the row sums of $R(m)$ are zero. More formally, we have that the $r_{i, j}(m)$ entry of $R(m)$ is defined as
\[
r_{i, j}(m) =
\begin{cases}
\rho_j(\mathrm{Frob}_{p_i}) &\text{if } i \neq j \\
\sum_{k \neq i} \rho_k(\mathrm{Frob}_{p_i}) &\text{if } i = j.
\end{cases}
\]
The usefulness of R\'edei matrices lies in the following theorem, which we quote from Stevenhagen's work \cite{Stevenhagen}, but originally goes back to R\'edei \cite{Redei}. It shows that the rank of the matrix $R(m)$ determines the $4$-rank of the narrow class group.

\begin{theorem}[\cite{Stevenhagen}]
For all square-free integers $m \neq 1$ we have
\[
\mathrm{rk}_4 \mathrm{Cl}^+(\Q(\sqrt{m})) = r - 1 - \mathrm{rk} \ R(m).
\]
\end{theorem}

\begin{remark}
Our R\'edei matrix $R(m)$ is the transpose of Stevenhagen's R\'edei matrix, but this does not affect the theorem statement.
\end{remark}

\subsection{A majorant for the 4-rank}
\label{ss4rank}
We will now construct a R\'edei majorant for the $4$-rank of class groups. As a first step, we construct the function $g(m, n)$.

\begin{definition}
\label{def:firstdefn}
Given an integer $a \neq 0$ and an integer $\alpha$ coprime to $a$, we will define a twisted matrix $R(a, \alpha)$. Let $a'$ be the square-free part of $a$ and let $q_1 < \dots < q_r$ be the odd prime divisors of $a'$. The twisted matrix $R(a, \alpha)$ has entries $r_{i, j}(a, \alpha)$ with
\[
r_{i, j}(a, \alpha) =
\begin{cases}
\chi_{q_j}(\mathrm{Frob}_{q_i}) &\text{if } i \neq j \\
\chi_\alpha(\mathrm{Frob}_{q_i}) + \sum_{k \neq i} \chi_{q_k}(\mathrm{Frob}_{q_i}) &\text{if } i = j.
\end{cases}
\]
Observe that $R(a, \alpha)$ is closely related to the matrix $R(a)$, except that the diagonal entries are twisted by the Legendre symbols corresponding to $\alpha$, that we have possibly removed the column and row corresponding to the prime $2$ and that we have used $\chi_{q_j}$ in place of $\chi_{q_j^\ast}$. Since $q_i$ is odd, we observe that $\chi_\alpha(\mathrm{Frob}_{q_i})$ is periodic in $\alpha$ with period $q_i$. Therefore we may define 
$$
g(m, n) := 2^{r - \mathrm{rk} \ R(m, n)} = |\ker(R(m, n))|,
$$
which depends only on $n \bmod{m}$.
\end{definition}

\begin{theorem}
\label{t4rankSub}
Let $k \in \Z_{\geq 1}$, and define $f_k(m) := 2^{k \cdot \mathrm{rk}_4 \mathrm{Cl}^+(\Q(\sqrt{m}))}$. Then we have
$$
f_k(mn) \leq g(m, n)^k 2^{k \omega(n) + k}
$$
for all non-zero coprime integers $m, n$. In particular, $f_k$ is $(4^k, g^k)$-R\'edei majorized.
\end{theorem}

\begin{proof}
We may assume without loss of generality that $m$ and $n$ are square-free. Looking at the definition of the R\'edei matrix $R(mn)$, we see that the right kernel of the R\'edei matrix $R(mn)$ naturally injects into the space
$$
V(m n) := \{\chi \in H^1(G_\Q, \FF_2) : \chi \cup \chi_{-mn} = 0, \chi \text{ ram. only at } 2 \infty mn\}.
$$
Indeed, this cup product detects whether the biquadratic extension cut out by $\chi$ and $\chi_{mn}$ lifts to a $D_4$-extension $L$ with the property that $\Gal(L/\Q(\sqrt{mn})) \cong \Z/4\Z$, i.e. $\Q(\sqrt{mn})$ sits in the middle of the field diagram for the resulting $D_4$-extension. The existence of such a lift is a necessary condition for $\chi$ to be a double in the class group.

Therefore we have that
$$
f_k(mn) = 2^{-k} |\ker(R(mn))|^k \leq 2^{-k} |V(m n)|^k.
$$
We consider the subspace of $V(mn)$ of codimension $2$ given by the basis $\chi_{q_1}, \dots, \chi_{q_r}$, where $q_1, \dots, q_r$ are the odd divisors of $mn$. Inspecting the local conditions of $\chi \cup \chi_{-mn}$ at the odd places and writing this down as a matrix, we see that $R(m, n)$ is a submatrix having dropped at most $\omega(n)$ rows and columns. Then the theorem follows from linear algebra.
\end{proof}

\subsection{Selmer matrices}
\label{ssSelmer}
Let $E$ be the elliptic curve given by the equation $y^2 = (x - r_1)(x - r_2)(x - r_3)$ for distinct integers $r_1, r_2, r_3$. We also assume that $\gcd(r_1, r_2, r_3)$ is square-free. Define $\delta_{i, j} := r_i - r_j$ and $\Omega := 2\delta_{1, 2} \delta_{1, 3} \delta_{2, 3}$. The primes dividing $\Omega$ include all finite places of bad reduction for $E$. Given $E$ and a positive, square-free integer $d$ coprime to $\Omega$, we define the twist
$$
E_d: dy^2 = (x - r_1)(x - r_2)(x - r_3).
$$
We shall require the following result about the $2$-Selmer group of $E_d$. Let $M := (\Z/2\Z)^2$. For a finite place $v \not \in \Omega$ and a square-free integer $d$, we define $\mathcal{L}_{d, v} \subseteq H^1(G_{\Q_v}, M) := \Q_v^\ast/\Q_v^{\ast 2} \times \Q_v^\ast/\Q_v^{\ast 2}$
$$
\mathcal{L}_{d, v} :=
\begin{cases}
H^1_{\text{nr}}(G_{\Q_v}, M) &\text{if } v(d) = 0 \\
\{(1, 1), (\delta_{12} \delta_{13}, d \delta_{12}), (d \delta_{21}, \delta_{21} \delta_{23}), (d \delta_{31}, d \delta_{32})\} &\text{if } v(d) = 1.
\end{cases}
$$
Note that $\mathcal{L}_{d, v}$ is a subgroup of $H^1(G_{\Q_v}, M)$. Writing $\mathbf{r} = (r_1, r_2, r_3)$, we define $\mathrm{Sel}_{\mathbf{r}}(M, d)$ as
$$
\mathrm{Sel}_{\mathbf{r}}(M, d) := \ker\left(H^1(G_\Q, M) \rightarrow \prod_{\substack{v \not \in \Omega \\ v \text{ finite}}} \frac{H^1(G_{\Q_v}, M)}{\mathcal{L}_{d, v}}\right).
$$

\begin{lemma}
\label{lSelmerBound}
Let $E$ be an elliptic curve of the shape $y^2 = (x - r_1)(x - r_2)(x - r_3)$ for distinct integers $r_1, r_2, r_3$ with $\gcd(r_1, r_2, r_3)$ square-free. Let $d$ be a positive integer coprime to $\Omega$. Then we have $\mathrm{Sel}^2(E_d) \subseteq \mathrm{Sel}_{\mathbf{r}}(M, d)$.

Moreover, suppose that the integers $r_1, r_2, r_3$ satisfy $\gcd(r_1, r_2, r_3) = 1$. In that case there exists a finite collection $\mathcal{C}$ of vectors $\mathbf{r}$ such that
$$
|\mathrm{Sel}^2(E_d)| \leq \max_{\mathbf{r} \in \mathcal{C}} |\mathrm{Sel}_{\mathbf{r}}(M, t)|
$$
for all square-free integers $d$, where $t$ is the largest positive divisor of $d$ coprime to $\Omega$.
\end{lemma}

\begin{proof}
The first part follows immediately from a standard $2$-descent, see Kane \cite[p.~1271]{Kane} or \cite[Section 7]{Watkins} for details. For the second part, one takes the collection $\mathcal{C}$ to be $(cr_1, cr_2, cr_3)$ for square-free integers $c$ all of whose prime divisors are in $\Omega$. Then the second part is a consequence of the first part.
\end{proof}

For $t$ coprime to $\Omega$, we now construct a linear operator with the eventual goal of writing $\mathrm{Sel}_{\mathbf{r}}(M, t)$ as the kernel of a matrix. The Selmer conditions $\mathcal{L}_{t, v}$ are self-dual with respect to the pairing
\[
((x_1, x_2), (x_1', x_2')) = (x_1, x_2')_v (x_2, x_1')_v.
\]
Suppose that $v(t) = 1$. Because the local conditions are self-dual, $(x_1, x_2)$ satisfies the local conditions at $v$ if and only if
$
(x_1, t \delta_{12})_v (x_2, \delta_{12} \delta_{13})_v = (x_1, \delta_{21} \delta_{23})_v (x_2, t \delta_{21})_v = 1.
$
We define $W$ to be the subspace of $H^1(G_\Q, M)$ unramified outside $\Omega$ and the primes dividing $t$. Concretely, we may view $W$ as pairs of square-free integers, of any sign, such that all prime divisors divide $\Omega \cdot t$.

Since $(x_1, x_2)$ has to be unramified for the places $v \not \in \Omega$ satisfying $v(t) = 0$, it is clear that $\mathrm{Sel}_{\mathbf{r}}(M, t) \subseteq W$. For the places with $v(t) = 1$, we define a linear map $\phi_v: W \rightarrow \mu_2^2$ given by
\begin{align}
\label{ePhiv}
(x_1, x_2) \mapsto \left((x_1, t \delta_{12})_v (x_2, \delta_{12} \delta_{13})_v, (x_1, \delta_{21} \delta_{23})_v (x_2, t \delta_{21})_v\right).
\end{align}
Then $\mathrm{Sel}_{\mathbf{r}}(M, t)$ is precisely the intersection, denoted $K$, of $\ker(\phi_v)$ among the $v$ satisfying $v(t) = 1$. Let $W'$ be the subspace of $W$ generated by $(x_1, x_2)$, where both $x_i$ consist of positive prime divisors of $t$. Then we have
\begin{align}
\label{eKernelBound}
|K| = \frac{|W' + K| |W' \cap K|}{|W'|} \leq \frac{|W|}{|W'|} |W' \cap K| \leq 4^{|\Omega| + 1} |W' \cap K|.
\end{align}
We are now ready to describe how to calculate $W' \cap K$ as the kernel of a square matrix. Write $t = p_1 \cdot \ldots \cdot p_r$ with $p_1 < \dots < p_r$. Consider the block matrix
$$
M'_\mathbf{r}(t) =
\begin{pmatrix}
A & D \\
D' & B
\end{pmatrix}
,
$$
where $D$ and $D'$ are diagonal matrices with
\[
D_{i, i} = \left(\frac{\delta_{12} \delta_{13}}{p_i}\right), \quad \quad D_{i, i}' = \left(\frac{\delta_{21} \delta_{23}}{p_i}\right),
\]
where our Legendre symbols take values in $\mathbb{F}_2$ (by identifying $\mathbb{F}_2$ with $\mu_2$) only for this subsection. Let us now describe the entries of $A$ and $B$, called $a_{i, j}$ and $b_{i, j}$ respectively. We have
$$
a_{i, j} =
\begin{cases}
\left(\frac{p_j}{p_i}\right) &\text{if } i \neq j \\
\left(\frac{\delta_{21}}{p_i}\right) + \sum_{k \neq i} \left(\frac{p_k}{p_i}\right) &\text{if } i = j
\end{cases}
$$
and
$$
b_{i, j} = 
\begin{cases}
\left(\frac{p_j}{p_i}\right) &\text{if } i \neq j \\
\left(\frac{\delta_{12}}{p_i}\right) + \sum_{k \neq i} \left(\frac{p_k}{p_i}\right) &\text{if } i = j.
\end{cases}
$$
With this construction we have that the right kernel of $M'_\mathbf{r}(t)$ is exactly $W' \cap K$. For a positive square-free integer $t$ coprime to $\Omega$, we define $f_\mathbf{r}(t)$ to be the size of $|\ker(M'_\mathbf{r}(t))|$. We extend $f_\mathbf{r}$ to all non-zero integers by the rules $f_\mathbf{r}(t) = f_\mathbf{r}(t p)$ for all $p$ dividing $\Omega$, $f_\mathbf{r}(t) = f_\mathbf{r}(-t)$ and $f_\mathbf{r}(t) = f_\mathbf{r}(t s^2)$.

More generally, given an integer $\alpha$ coprime to $t$, we construct a matrix $M'_\mathbf{r}(t, \alpha)$ of the shape
$$
M'_\mathbf{r}(t, \alpha) =
\begin{pmatrix}
A_\alpha & D \\
D' & B_\alpha
\end{pmatrix}
,
$$
where $D$ and $D'$ are the same matrices as before, and $A_\alpha$ and $B_\alpha$ are given by
$$
a_{i, j, \alpha} =
\begin{cases}
\left(\frac{p_j}{p_i}\right) &\text{if } i \neq j \\
\left(\frac{\alpha \delta_{21}}{p_i}\right) + \sum_{k \neq i} \left(\frac{p_k}{p_i}\right) &\text{if } i = j
\end{cases},
\quad \quad
b_{i, j, \alpha} = 
\begin{cases}
\left(\frac{p_j}{p_i}\right) &\text{if } i \neq j \\
\left(\frac{\alpha \delta_{12}}{p_i}\right) + \sum_{k \neq i} \left(\frac{p_k}{p_i}\right) &\text{if } i = j.
\end{cases}
$$
For a positive square-free integer $t$ coprime to $\Omega$ and an integer $\alpha$ coprime to $t$, we define $g_\mathbf{r}(t, \alpha)$ to be the size of the kernel of $M'_\mathbf{r}(t, \alpha)$. We extend this to all non-zero integers $d$ and all integers $\alpha$ coprime to $d$ by demanding that 
\begin{equation}
\label{def:stackbundles}
g_\mathbf{r}(d, \alpha) = g_\mathbf{r}(t, \alpha)
\end{equation} 
with $t$ the largest square-free divisor of $d$ that is coprime to $\Omega$.

\begin{theorem}
\label{tSelmerSub}
Let $E$ be an elliptic curve of the shape $y^2 = (x - r_1)(x - r_2)(x - r_3)$ with $r_1, r_2, r_3$ distinct integers satisfying $\gcd(r_1, r_2, r_3) = 1$. Let $k \in \Z_{\geq 1}$ and let $\mathcal{C}$ be the collection from Lemma \ref{lSelmerBound}. Then we have
$$
|\mathrm{Sel}^2(E_d)|^k \leq 4^{k \cdot |\Omega| + k} \max_{\mathbf{r} \in \mathcal{C}} f_\mathbf{r}(d)^k
$$
for all non-zero integers $d$, and
$
f_\mathbf{r}(mn)^k \leq g_\mathbf{r}(m, n)^k 4^{k \cdot \omega(n)}
$
for all non-zero coprime integers $m, n$. In particular, $f_\mathbf{r}^k$ is $(4^k, g_{\b r}^k)$-R\'edei majorized.
\end{theorem}

\begin{proof}
To prove the first inequality, we may reduce to the case that $d$ is square-free by definition of $E_d$ and $f_\mathbf{r}(d)$. Then Lemma \ref{lSelmerBound} and equation \eqref{eKernelBound} confirm the validity of
$$
|\mathrm{Sel}^2(E_d)|^k \leq \max_{\mathbf{r} \in \mathcal{C}} |\mathrm{Sel}_\mathbf{r}(M, t)|^k \leq 4^{k \cdot |\Omega| + k} \max_{\mathbf{r} \in \mathcal{C}} f_\mathbf{r}(t)^k = 4^{k \cdot |\Omega| + k} \max_{\mathbf{r} \in \mathcal{C}} f_\mathbf{r}(d)^k
$$
with $t$ the largest positive divisor of $d$ coprime to $\Omega$.

To prove the second inequality, we may assume without loss of generality that $m$ and $n$ are square-free. Since the matrix $M'_\mathbf{r}(m, n)$ is a submatrix of $M'_\mathbf{r}(mn)$ obtained by adding at most $2\omega(n)$ rows and columns, the result follows.
\end{proof}

\subsection{Level of distribution results}
In this subsection we state the level of distribution results that we will use for the sieving process. Our results in this subsection are not optimal but will suffice for our purposes. Let $\delta(m)$ be the multiplicative function satisfying
\begin{align}
\label{eh3Density}
\delta(p^e) = 
\begin{cases}
\frac{1}{p + 1}, &\text{if } p \geq 2 \text{ and } e = 1, \\
0, &\text{if } p > 2 \text{ and } e \geq 2, \\
\frac{1}{3}\mathds 1_{\{2\}}(e)+\frac{1}{6}\mathds 1_{\{3\}}(e), &\text{if } p = 2 \text{ and } e \geq 2.
\end{cases}
\end{align}

\begin{lemma}
\label{lh3Level}
Let $m \in \Z_{\geq 1}$ and let $q_1 < \dots < q_r$ be the odd prime divisors of $m$. Let $S$ be a subset of $\{1, \dots, r\}$, and for each $i \in S$, let $\epsilon_i \in \{\pm 1\}$. Then we have
\begin{align*}
\sum_{\substack{0 < \Delta(n) < X, m\mid n \\ \left(\frac{n/m}{q_i}\right) = \epsilon_i \ \forall i \in S}} (h_3(n) - 1) 
= \frac{X \delta(m)}{2^{|S|} \pi^2} + O(X^{6/7})
\end{align*}
uniformly for all $m \leq X^{1/100}$, and similarly
\begin{align*}
\sum_{\substack{0 < -\Delta(n) < X, m\mid n \\ \left(\frac{n/m}{q_i}\right) = \epsilon_i \ \forall i \in S}} (h_3(n) - 1) 
= \frac{3X \delta(m)}{2^{|S|} \pi^2} + O(X^{6/7})
\end{align*}
uniformly for all $m \leq X^{1/100}$.
\end{lemma}
 
\begin{proof}
Let us prove the first part of the lemma, the second part may be proven by an identical procedure. For a prime $p$, we let $\Sigma_p$ be a set of (isomorphism classes of) \'etale cubic algebras over $\Q_p$. Given a sequence $\Sigma = (\Sigma_p)_p$, we define $N_3(X, \Sigma)$ to be the set of cubic fields with $0 < \mathrm{Disc}(F) < X$ such that $F \otimes \Q_p \in \Sigma_p$ for all $p$. We write $A_p$ for the set of all \'etale cubic algebras and we write $A_p'$ for the set of all \'etale cubic algebras that are not totally ramified. We call a local specification $\Sigma$ valid if the set of primes $p$ for which $\Sigma_p \neq A_p, A_p'$ is finite. Then \cite[Theorem 1.3]{BTT} shows that
\begin{align}
\label{eBTT}
N_3(X, \Sigma) = \frac{X}{12 \zeta(3)} \prod_p C_p(\Sigma_p) + O\left(2^\kappa X^{5/6}\right),
\end{align}
where $\kappa$ equals the number of places for which $\Sigma_p \neq A_p, A_p'$. For each $i = 1, \dots, r$, let $t_i$ be a square in $\Q_p^\ast$ if $\epsilon_i = 1$ and let $t_i$ be a non-square unit in $\Q_p^\ast$ if $\epsilon_i = -1$. Our lemma is clear if $m$ is divisible by $p^2$ for some $p \geq 3$ or if $m$ is divisible by $16$. Otherwise, we apply equation \eqref{eBTT} with the following valid local specification $\Sigma = (\Sigma_p)_p$:
$$
\Sigma_p = 
\begin{cases}
\Q_p \times \Q_p(\sqrt{p t_i}) &\text{if } p = q_i \text{ for some } i \in S \\
\{\Q_2 \times \Q_2(\sqrt{x}) : x \in \{-1, -5, 2, -2, 10, -10\}\} &\text{if } p = 2 \text{ and } v_2(m) \in \{1, 2\} \\
\{\Q_2 \times \Q_2(\sqrt{x}) : x \in \{2, -2, 10, -10\}\} &\text{if } p = 2 \text{ and } v_2(m) = 3 \\
A_p' &\text{otherwise.}
\end{cases}
$$
With this local specification we have by construction
$$
\sum_{\substack{0 < \Delta(n) < X \\ n \equiv 0 \bmod m \\ \left(\frac{n/m}{q_i}\right) = \epsilon_i \ \forall i \in S}} (h_3(n) - 1) = 2N_3(X, \Sigma).
$$
Using \cite[Table 1]{BTT}, the formula $C_2(E) = \frac{6c_2}{7}$ in \cite[Section 1]{BTT} for a partially ramified cubic \'etale $\Q_2$-algebra and the formulas for $c_2$ in \cite[Section 8]{BTT}, one computes
$$
C_p(\Sigma_p) =
\begin{cases}
\frac{p}{2p^2 + 2p + 2}, &\text{if } p = q_i \text{ for some } i \in S \\
\frac{2}{7} \mathds 1_{\{1,2\}}(v_2(m))+\frac{1}{7} \mathds 1_{\{3\}}(v_2(m)), &\text{if } p = 2 \text{ and } v_2(m) \in \{1, 2\} \\
\frac{p^2 + p}{p^2 + p + 1}, &\text{otherwise.}
\end{cases}
$$
This concludes the proof by writing $\zeta(3)^{-1} = \prod_p (1 - p^{-3})$, multiplying the Euler products and recognizing that $6/\pi^2 = \zeta(2)^{-1} = \prod_p (1 - p^{-2})$.
\end{proof}

Given an integer $n \geq 1$, a divisor $a$ of $n$ and an element $x \in \Z/n\Z$ satisfying $x \equiv 0 \bmod a$, we define $x/a$ to be the unique element of $\Z/(n/a)\Z$ that maps to $x$ under the multiplication by $a$ map. Furthermore, if $q$ is an odd prime dividing some integer $n$, we define for any $a \in \Z/n\Z$ the Legendre symbol $(a/q)$ to be the unique integer in $\{1, -1, 0\}$ such that
\[
\left(\frac{a}{q}\right) \equiv a^{\frac{q - 1}{2}} \bmod q.
\]
Let $a \geq 1$ be an integer and let $q_1, \dots, q_r$ be the odd prime divisors of $a$. For each $i \in \{1, \dots, r\}$, choose $\epsilon_i \in \{1, -1, 0\}$ and let $\bm{\epsilon} = (\epsilon_i)_{1 \leq i \leq r}$ be the resulting vector. Define
\[
h(a, \bm{\epsilon}) := \frac{\left|\left\{(t_1, \dots, t_n) \in (\frac{\Z}{a q_1 \cdot \ldots q_r \Z})^n : P(t_1, \dots, t_n) \equiv 0 \bmod a, \left(\frac{P(t_1, \dots, t_n)/a}{q_i}\right) = \epsilon_i\right\}\right|}{a^n q_1^n \cdot \ldots \cdot q_r^n}
\]
and
\[
h(a) = \sum_{\bm{\epsilon}} h(a, \bm{\epsilon}) = \frac{\left|\left\{(t_1, \dots, t_n) \in (\frac{\Z}{a \Z})^n : P(t_1, \dots, t_n) \equiv 0 \bmod a\right\}\right|}{a^n}.
\]

\begin{lemma}
\label{lPolyLevel}
Let $P \in \Z[t_1, \dots, t_n]$ be a separable polynomial of degree at least $1$. Then there exists $\theta > 0$, depending only on the degree of $P$, and $C > 0$, depending only on $P$, such that
uniformly for all $B \geq 1$,
all $ a \leq B^{\theta}$ and $\bm{\epsilon} = (\epsilon_i)_{1 \leq i \leq r}$ we have 
$$
\left|\#\l\{\b t \in \Z^n: \max_i |t_i| \leq B, a\mid P(\b t ), \left(\frac{P(t_1, \dots, t_n)/a}{q_i}\right) = \epsilon_i\r\}
 - h(a, \bm{\epsilon}) \cdot (2B)^n\right| \leq C B^{n - \theta}
.$$
Moreover, there are constants $C_1, \dots, C_5>0$ depending only on $P$ such that 
\begin{enumerate}
\item[(i)] 
$
h(p, \epsilon) \leq h(p) \leq C_1/p
$ 
for all odd primes $p$ and all $\epsilon$;
\item[(ii)] 
$
h(p^e, \epsilon) \leq h(p^e) \leq C_2/p^2
$ 
for all $e \geq 2$, all odd $p$ and all $\epsilon$;
\item[(iii)] 
$
h(p^e, \epsilon) \leq h(p^e) \leq C_3/p^{eC_4}
$ 
for all $e \geq 1$, all odd primes $p$ and all $\epsilon$;
\item[(iv)] for all odd primes $p$ and all $\epsilon \in \{\pm 1\}$ we have 
$$
\left|h(p, \epsilon) - \frac{h(p)}{2}\right| \leq \frac{C_5}{p^2}.
$$
\end{enumerate}
\end{lemma}

\begin{proof}
The first part of the lemma follows upon fixing congruence classes for the variables $t_1, \dots, t_n$ modulo $a \cdot q_1 \cdot \ldots \cdot q_r$ and covering the cube $\max_i |t_i| \leq B$ with boxes.

For the second part of the lemma, we always have $h(p^e, \epsilon) \leq h(p^e)$. Bound $(i)$ follows from the Lang--Weil \cite{LW} bound. Bound $(ii)$ and $(iii)$ follow from respectively \cite[Lemma 2.6]{CKPS} and \cite[Lemma 2.8]{CKPS2}. It remains to prove bound $(iv)$. Note that we may assume without loss of generality that $p$ is larger than any given constant $C'$ depending only on $P$. In particular, by taking $C'$ large enough, we ensure that $p$ is odd and that the reduction of $P$ in $\mathbb{F}_p[t_1, \dots, t_n]$ remains separable. 

We define $\mathcal{Z}(P)$ to be the set of $(c_1, \dots, c_n)$ satisfying $P(c_1, \dots, c_n) \equiv 0 \bmod p$, and we define $\mathrm{Hen}(P)$ to be the subset of $\mathcal{Z}(P)$ for which there exists some $i$ such that
\[
\frac{\partial P}{\partial t_i} (c_1, \dots, c_n) \not \equiv 0 \bmod p.
\]
Since $P$ is separable, the system
\[
P(c_1, \dots, c_n) \equiv 0 \bmod p, \ \frac{\partial P}{\partial t_1} (c_1, \dots, c_n) \equiv 0 \bmod p, \ \dots, \ \frac{\partial P}{\partial t_n} (c_1, \dots, c_n) \equiv 0 \bmod p
\]
has codimension at least $2$. Appealing to the Lang--Weil \cite{LW} bounds, we may therefore bound the contribution from $\mathcal{Z}(P) - \mathrm{Hen}(P)$. For the points in $\mathrm{Hen}(P)$, we write every element of $\Z/p^2\Z$ as $c_i + d_ip$ with $0 \leq c_i, d_i \leq p - 1$. Using Taylor expansion around $(c_1, \dots, c_n)$ as in the proof of Hensel's lemma demonstrates the validity of
$$
P(c_1 + d_1p, \dots, c_n + d_np) \equiv P(c_1, \dots, c_n) + p \cdot \left(\sum_{i = 1}^n d_i \cdot \frac{\partial P}{\partial t_1} (c_1, \dots, c_n)\right) \bmod{p^2}.
$$
Therefore given any point $(c_1, \dots, c_n) \in \mathrm{Hen}(P)$, exactly $\frac{p^{n - 1}(p - 1)}{2}$ lifts will contribute to $h(p, \epsilon)$. Using the bound $(i)$, this readily gives the lemma.
\end{proof}

\section{Sieving}
\label{sNT}
This section adapts previous work of Nair--Tenenbaum \cite{NT} and Wolke \cite{Wolke}, which significantly strengthened and generalized old work of Erd\H{o}s \cite{Erdos}, Shiu~\cite{Shiu} and Nair~\cite{nair}. Our previous related work in this direction~\cite{CKPS} is not flexible enough; see Remark~\ref{rem:unfairpkoln}.

\subsection{Main sieve argument} 
We start by introducing the sequences to which our main sieve theorem applies.

\begin{definition}
\label{dClasses}
Let $\kappa, \lambda, K > 0$, $B \geq 3$ be real numbers. We say that a multiplicative function $h: \Z_{\geq 1 } \rightarrow [0 , \infty)$ belongs to the class $\mathcal{D}(\kappa, \lambda, B, K)$ if
\begin{itemize}
\item for all $B < w < z$ we have
\begin{align}
\label{eFSieveAss}
\prod_{w \leq p < z} (1 - h(p))^{-1} \leq \left(\frac{\log z}{\log w}\right)^\kappa \left(1 + \frac{K}{\log w}\right),
\end{align}
\item for every prime $p > B$ and $e \in \Z_{\geq 1}$
\begin{align}
\label{ehOnPrimes}
h(p^e) \leq \frac{B}{p},
\end{align}
\item for every prime $p$ and $e \in \Z_{\geq 1}$
\begin{align}
\label{ehDecay}
h(p^e) \leq B p^{-e \lambda}.
\end{align}
\end{itemize} Also pick for each prime power $p^e$ a partition $\mathcal{P}(p^e) = \{\mathcal{A}_1, \dots, \mathcal{A}_k\}$ of $\Z/p^e\Z$. We demand that for all $i$ we have 
$
\mathcal{A}_i \subseteq (\Z/p^e\Z)^\ast$ or $\mathcal{A}_i \subseteq \Z/p^e\Z - (\Z/p^e\Z)^\ast.
$
This naturally gives partitions $\mathcal{P}(m)$ of $\Z/m\Z$ for each integer $m$ by taking the product sets of the resulting partitions over the prime powers exactly dividing $m$. We call $ \overline{\c P}$ the collection of partitions $\c P(m)$ as $m$ varies.

Let $(a_n)_{n \geq 1}$ be a sequence of strictly positive integers and let $(w_n)_{n \geq 1}$ be a sequence of non-negative real numbers. Fix positive constants $\alpha$ and $\theta$. We say that $(a_n)_{n \geq 1}$ belongs to the class $\mathcal{C}(\alpha, \theta, \kappa, \lambda, B, K)$ with weights $(w_n)_{n \geq 1}$ and size function $M: [1, \infty) \rightarrow [1, \infty)$ if
\begin{itemize}
\item the function $M$ is non-decreasing and goes to infinity,
\item we have $a_n \leq \alpha M(n)^\alpha$,
\item we require that there exists some non-negative function $h(\cdot, \cdot)$ such that
\begin{align}
\label{eLevel}
\left|\sum_{\substack{ n \leq X, a_n/m \in \mathcal{A} \\ a_n \equiv 0 \bmod m}} w_n
-h(m, \mathcal{A}) M(X) 
\right| \leq K M(X)^{1 - \theta}
\end{align}
uniformly for all $X \geq 1$, all $m \leq M(X)^\theta$ and all $\mathcal{A} \in \mathcal{P}(m)$,
\item the function $h(m) = \sum_{\mathcal{A} \in \mathcal{P}(m)} h(m, \mathcal{A})$ lies in the class $\mathcal{D}(\kappa, \lambda, B, K)$, and moreover
\begin{align}
\label{eLevel23}
h(m, \mathcal{A}) h(n, \mathcal{B}) = h(mn, \mathcal{A} \times \mathcal{B})
\end{align}
for all coprime $m$ and $n$ and all $\mathcal{A} \in \mathcal{P}(m)$, $\mathcal{B} \in \mathcal{P}(n)$,
\item defining
$$
H(d) := 
\begin{cases}
\sum_{\mathcal{A} \in \mathcal{P}(d)} g(d, \mathcal{A}) \frac{h(d, \mathcal{A})}{h(d)}, &\textup{if } h(d) \neq 0, \\
0, &\textup{if } h(d) = 0,
\end{cases}
$$
we assume that for every $\epsilon > 0$, there exists $C_\epsilon > 0$ such that for all coprime $d_1, d_2$ we have 
\begin{align}
\label{eHSubm}
H(d_1 d_2) \leq H(d_1) \min(K^{\Omega(d_2)}, C_\epsilon d_2^\epsilon).
\end{align} 
\end{itemize}
\end{definition}

The reason for the rather general formulation with partitions in Definition \ref{dClasses} is that Lemma \ref{lh3Level} is only able to detect whether $n/m$ is a square, but not the precise class modulo $m$. It is highly plausible that one can directly get a level of distribution for $h_3 \equiv t \bmod m$ (see the discussion in \cite{BF} for example) but we do not know of such a result in the literature. In that case one could take all the partitions to be the one element subsets of $\Z/p^e\Z$. In any case, we believe that the flexibility allowed in Definition \ref{dClasses} may be valuable for future applications as well. 

We say that a function $g: \{(m, n) \in \Z_{\geq 1}^2 : \gcd(m, n) = 1\} \rightarrow[0 , \infty)$ is compatible with 
$ \overline{\c P}$
if
$$
g(m, n) = g(m, n')
$$
for all $m \in \Z_{\geq 1}$, all $\mathcal{A} \in \mathcal{P}(m)$ and all $n, n'$ coprime to $m$ satisfying $n \bmod m \in \mathcal{A}$ and $n' \bmod m \in \mathcal{A}$. This allows us to define $g(m, \mathcal{A}) := g(m, n)$ for any choice of $n \in \mathcal{A}$. It will be convenient to define $g(m, n) := 0$ if $m$ and $n$ are not coprime.

\begin{theorem}
\label{tNT}
Let $\alpha, \theta, \kappa, \lambda, K > 0$, $B \geq 3$ be real numbers. Fix $A > 1$ and fix a function $g: \{(m, n) \in \Z_{\geq 1}^2 : \gcd(m, n) = 1\} \rightarrow [0 , \infty)$ satisfying~\eqref{ePeriod}. Let $ \overline{\c P}$ be a collection of partitions constructed as above. Let $f: \Z_{\geq 1} \rightarrow [0 , \infty)$ be $(A, g)$-R\'edei majorized. Assume that $g$ is compatible with $ \overline{\c P}$.
We also assume that for every $\epsilon > 0$, there exists $C'_\epsilon > 0$ with
\begin{align}
\label{eGXbound}
\max_{1 \leq d \leq M(X)} \max_{\mathcal{A} \in \mathcal{P}(d)} g(d, \mathcal{A}) \leq C'_\epsilon M(X)^\epsilon.
\end{align}
Then there exists $C > 0$ such that for all sequences $(a_n)_{n \geq 1}$ belonging to the class $\mathcal{C}(\alpha, \theta, \kappa, \lambda, B, K)$ with weights $(w_n)_{n \geq 1}$ and size function $M$, and all $X \geq 1$ satisfying $M(X) \geq C$ we have 
$$
\sum_{1 \leq n \leq X} w_n f(a_n) \leq C M(X) \prod_{B < p \leq M(X)} (1 - h(p)) \sum_{1 \leq d \leq M(X)} \sum_{\mathcal{A} \in \mathcal{P}(d)} g(d, \mathcal{A}) h(d, \mathcal{A}).
$$
\end{theorem}

\begin{remark}
Tracing through the proof, one finds that $C$ may be chosen to depend only on $A, \alpha, \theta, \kappa, \lambda, B, K, g(1, 0)$ and the constants
$C_\epsilon,C'_\epsilon$ in \eqref{eHSubm} and \eqref{eGXbound} for some $\epsilon > 0$ depending only on the parameters $A, \alpha, \theta, \kappa, \lambda, B, K$.
\end{remark}

\begin{remark}\label{rem:unfairpkoln}Theorem 1.9 in~\cite{CKPS} does not cover Theorem~\ref{tNT}
as it has no flexibility regarding the partitions $\c P$
and property~\eqref{eHSubm}.
\end{remark}

\begin{proof}
We let $ \eta_1, \eta_2$ be positive constants that we shall choose later in terms of $\alpha, \theta, \kappa, \lambda, B, K$ and take $Z := M(X)^{\eta_1}$. Factor $a_n = p_1^{e_1} \cdots p_r^{e_r}$ with $r \geq 0$ primes $p_1 < \dots < p_r$ and exponents $e_i \geq 1$. Let $i$ be the largest index 
with $p_1^{e_1} \cdots p_i^{e_i} \leq Z$ and set $c_n := p_1^{e_1} \cdots p_i^{e_i}, b_n := a_n/c_n$. Thus, 
\begin{equation}\label{ebncngcd}
 P^+(c_n) < P^-(b_n), 
 \gcd(b_n, c_n) = 1 \textrm{ and } 
 c_n \leq Z. 
\end{equation}
The following cases are mutually exclusive and cover all scenarios:
\begin{itemize}
\item[(i)] $P^-(b_n) \geq Z^{\eta_2}$,
\item[(ii)] $P^-(b_n) < Z^{\eta_2}$ and $c_n \leq Z^{1/2}$,
\item[(iii)] $P^-(b_n) \leq (\log Z)( \log \log Z)$ and $Z^{1/2} < c_n \leq Z$,
\item[(iv)] $(\log Z) (\log \log Z)< P^-(b_n) < Z^{\eta_2}$ and $Z^{1/2} < c_n \leq Z$.
\end{itemize}
The constants $C_1, C_2, \dots$ appearing in the proof will depend at most on $\alpha, \theta, \kappa, \lambda, B, K$ and $\eta_1, \eta_2$.

\subsection*{Case (i)}
The plan is to show that $b_n$ has a bounded number of prime divisors. Once we prove this, we will be able to replace $f(a_n)$ by $g(c_n, b_n)$ while only losing a constant. We will then be able to employ the Brun sieve to bound the number of $c_n$ arising from some $a_n$ in this way.

Since $a_n \leq \alpha M(n)^\alpha \leq \alpha M(X)^\alpha$ for $n \leq X$ and $P^-(b_n) \geq Z^{\eta_2}$, there exists a constant $C_1 > 0$ such that $\Omega(b_n) \leq C_1$ for $M(X) \geq C_1$. Using~\eqref{ebncngcd} and that $f$ is $(A, g)$-R\'edei majorized we deduce the inequality $f(a_n) = f(b_n c_n) \leq A^{C_1} g(c_n, b_n)$. We set $d := c_n$, so that $d \leq Z$ and $d \mid a_n$. Because we are in case (i), it follows that $a_n/d$ is coprime to every prime in the interval $[2, Z^{\eta_2})$. In particular, $a_n$ is coprime to every prime in the interval $(B, Z^{\eta_2})$ not dividing $d$. Put
$$
P := \prod_{\substack{p \in (B, Z^{\eta_2}) \\ p \nmid d}} p.
$$ 
Hence, the contribution of case (i) towards the sum over $n$ in Theorem~\ref{tNT} is at most 
\begin{align}
\label{eCasei1}
A^{C_1} \sum_{d \leq Z} \sum_{\substack{1 \leq n \leq X, d \mid a_n \\ \gcd(P, a_n) = 1}} w_n g(d, b_n) = A^{C_1} \sum_{d \leq Z} \sum_{\mathcal{A} \in \mathcal{P}(d)} g(d, \mathcal{A}) \sum_{\substack{1 \leq n \leq X, d \mid a_n \\ \gcd(P, a_n) = 1}} w_n \mathds 1_{\c A}(a_n/d),
\end{align}
since $g$ is compatible with $\overline{\mathcal{P}}$. Taking $y = Z$ in the Fundamental lemma of Sieve Theory~\cite[Lemma 6.3]{iwaniec}, there exists a sequence of real numbers $(\lambda_m^+)$ depending only on $\kappa$ such that
\begin{alignat*}{2}
&\lambda_1^+ = 1, && \hspace{1cm} |\lambda_m^+| \leq 1 \ \ \ \quad \text{ if } 1 < m < Z, \\
&\lambda_m^+ =0 \quad \text{ if } m \geq Z, && \hspace{1cm} 0 \leq \sum_{m \mid a} \lambda_m^+ \ \ \ \text{ for } a > 1.
\end{alignat*}
Moreover, for any multiplicative function $f(m)$ with $0 \leq f(p) < 1$ satisfying
\begin{align}
\label{eFLSieveAssu}
\prod_{w \leq p < z} (1 - f(p))^{-1} \leq \left(\frac{\log z}{\log w}\right)^\kappa \left(1 + \frac{K}{\log w}\right)
\end{align}
for all $2 \leq w < z \leq Z$, we have
\begin{align}
\label{eFundamentalLemma}
\sum_{m \mid P(z)} \lambda_m^+ f(m) = \left(1 + O\left(e^{-\sigma} \left(1 + \frac{K}{\log z}\right)^{10}\right)\right) \prod_{p \leq z} (1 - f(p)),
\end{align}
where $P(z)$ is the product of all primes $p \leq z$ and $\sigma = \log Z/\log z$.

We continue to upper bound the inner sum over $n$ in the right-hand side of~\eqref{eCasei1} by 
\begin{align}
\label{eCasei2}
\sum_{\substack{1 \leq n \leq X, d \mid a_n \\ a_n/d \in \mathcal{A}}} w_n \sum_{\substack{m \mid a_n \\ m \mid P}} \mu(m) \leq \sum_{\substack{1 \leq n \leq X, d \mid a_n \\ a_n/d \in \mathcal{A}}} w_n \sum_{\substack{m \mid a_n \\ m \mid P}} \lambda_m^+ = \sum_{m \mid P} \lambda_m^+ \sum_{\substack{1 \leq n \leq X, dm \mid a_n \\ a_n/d \in \mathcal{A}}} w_n. 
\end{align}
Using the partitions $\c P(m)$ and~\eqref{eLevel} the right-hand side becomes 
$$ 
\sum_{m \mid P} \lambda_m^+ \sum_{\c B\in \c P(m)} \left(h(d m, \mathcal{A}\times \c B) M(X) + O( M(X)^{1 - \theta})\right) 
$$ 
because we can ensure $dm \leq Z^2 \leq M(X)^\theta$ by using that $\lambda_m^+$ is supported on $[1, Z)$ and taking $\eta_1\leq \theta/2$. Exploiting $\sum_{\c B\in \c P(m)} |\c B| = m\leq Z$ for the error term and~\eqref{eLevel23} for the main term we get 
\begin{align}
\label{eCasei3}
h(d, \mathcal{A}) M(X) \sum_{m \mid P} \lambda_m^+ h(m) + O(Z^2 M(X)^{1 - \theta}).
\end{align} 
By~\eqref{eFundamentalLemma} with $f(p) = h(p) \mathds{1}_{p > B} \mathds{1}_{p \nmid d}$, there exists $C_2 > 0$ such that
\begin{align}
\label{eCasei4}
\sum_{m \mid P} \lambda_m^+ h(m) \leq C_2 \prod_{\substack{B < p < Z^{\eta_2} \\ p \nmid d}} (1 - h(p)).
\end{align}
The conditions \eqref{eFLSieveAssu} and $0 \leq f(p) < 1$ follow immediately from assumptions~\eqref{eFSieveAss} 
and~\eqref{ehOnPrimes}. Furthermore, we may extend the product in equation \eqref{eCasei4} to all $B < p \leq M(X)$ at the expense of losing a constant due to~ \eqref{eFSieveAss}. Gathering \eqref{eCasei1}, \eqref{eCasei2}, \eqref{eCasei3} and \eqref{eCasei4}, we conclude 
\begin{align*}
\sum_{\substack{n \leq X \\ \text{case (i)}}} w_n f(a_n) &\leq A^{C_3} \sum_{d \leq Z} \sum_{\mathcal{A} \in \mathcal{P}(d)} g(d, \mathcal{A}) 
\Big(h(d, \mathcal{A}) M(X) \prod_{\substack{B < p \leq M(X) \\ p \nmid d}} (1 - h(p)) + Z^2 M(X)^{1 - \theta}\Big) 
\\&\leq A^{C_3} M(X) \sum_{d \leq Z} \sum_{\mathcal{A} \in \mathcal{P}(d)} \hspace{-0.175cm} g(d, \mathcal{A}) h(d, \mathcal{A}) \hspace{-0.4cm} \prod_{\substack{B < p \leq M(X) \\ p \nmid d}} \hspace{-0.4cm} (1 - h(p)) + A^{C_3} C_{\theta/2} 
Z^4 M(X)^{1 - \theta/2 },
\end{align*} for some $C_3 > 0$ by~\eqref{eGXbound} with $\epsilon = \theta/2$. We rewrite the first term as
$$
\sum_{d \leq Z} \sum_{\mathcal{A} \in \mathcal{P}(d)} g(d, \mathcal{A}) h(d, \mathcal{A}) \prod_{\substack{B < p \leq M(X)\\ p \nmid d}} (1 - h(p)) = \prod_{B < p \leq M(X)} (1 - h(p)) \sum_{d \leq Z} h(d) H(d) \prod_{\substack{ p \mid d \\ p>B }} (1 - h(p))^{-1}.
$$
We now apply \cite[Lemma 2.7]{CKPS} with the choices $h = F$ and $G = H$. The conditions on $F$ in that lemma are satisfied thanks to equations \eqref{ehOnPrimes} and \eqref{ehDecay}. The condition on $G$ in that lemma is satisfied thanks to equation \eqref{eHSubm}. By positivity we may also extend the sum over $d \leq Z$ to all $d \leq M(X)$. These manipulations transform our final upper bound for case (i) to
\begin{align}
\label{eConclusioni}
\ll M(X) \prod_{B < p \leq M(X)} (1 - h(p)) \sum_{1 \leq d \leq M(X)} \sum_{\mathcal{A} \in \mathcal{P}(d)} g(d, \mathcal{A}) h(d, \mathcal{A}) + Z^4 M(X)^{1 - \theta/2}. 
\end{align} 
 
\subsection*{Case (ii)}
It will be shown that the exponent of $P^-(b_n)$ in the prime factorization of $b_n$ is large and that can only happen 
very rarely. Let $q := P^-(b_n)$. The definition of case (ii) and of $b_n$ shows that $Z < c_n q^{v_q(b_n)}$ and $c_n \leq Z^{1/2}$, thus, $ Z^{1/2} < q^{v_q(b_n)}$. We let $f_q$ be the largest positive integer such that $q^{f_q} \leq M(X)^\theta$ and $f_q \leq v_q(b_n)$ . Since we already assumed $2\eta_1 \leq \theta$, we have
\begin{align}
\label{eqfqControl}
q^{f_q} > \frac{M(X)^{\min(\theta, \eta_1/2)}}{q} = \frac{M(X)^{\eta_1/2}}{q} > \frac{M(X)^{\eta_1/2}}{Z^{\eta_2}} = M(X)^{\frac{\eta_1}{2} - \eta_1 \eta_2}
\end{align}
by the assumption $q < Z^{\eta_2}$ of case (ii). Therefore, we have found an integer $f_q$ for each prime $q < Z^{\eta_2}$ with the properties $q^{f_q} \mid q^{v_q(b_n)} \mid b_n \mid a_n$, $q^{f_q} \leq M(X)^\theta$ and \eqref{eqfqControl}. We now bring the R\'edei majorant of $f$ into play by taking the second term in the minimum of \eqref{efWSub}. Combining this with~\eqref{eGXbound} and the bound $a_n \leq \alpha M(n)^\alpha \leq \alpha M(X)^\alpha$ leads us to the estimate
$$
\sum_{\substack{1 \leq n \leq X \\ \text{case (ii)}}} w_n f(a_n) \leq C_4(\gamma, \epsilon) M(X)^{\gamma + \epsilon}
\sum_{\substack{1 \leq n \leq X \\ \text{case (ii)}}} w_n \leq C_4(\gamma, \epsilon) M(X)^{\gamma + \epsilon}
\sum_{q < Z^{\eta_2}} \sum_{\substack{1 \leq n \leq X \\ q^{f_q} \mid a_n }} w_n,
$$
where $\gamma$ and $\epsilon$ will be chosen later in terms of $\theta, \lambda, \eta_1, \eta_2$ at the end of case (ii). This latter sum can be estimated by alluding to our level of distribution assumption \eqref{eLevel} and by using the arguments involving the partitions $\c B \in \c P(m)$ proving~\eqref{eCasei3}. The resulting bound is 
$$ 
\ll \sum_{q < Z^{\eta_2}} \left(h(q^{f_q}) M(X) + Z M(X)^{1 - \theta}\right) \leq M(X) \sum_{q < Z^{\eta_2}} h(q^{f_q}) + Z^{1 + \eta_2} M(X)^{1 - \theta}.
$$
Finally, we employ \eqref{ehDecay}, \eqref{eqfqControl} and the construction of $f_q$ to bound
$$
\sum_{q < Z^{\eta_2}} h(q^{f_q}) \leq B \sum_{q < Z^{\eta_2}} q^{-f_q \lambda} < B M(X)^{-\lambda(\frac{\eta_1}{2} - \eta_1 \eta_2)} \sum_{q < Z^{\eta_2}} 1 \leq B M(X)^{\frac{-\lambda \eta_1}{2} + \lambda \eta_1 \eta_2 + \eta_1 \eta_2}.
$$
Picking $\eta_2$ and $\gamma$ in such a way that
\begin{align}
\label{eParChoices}
1 + \eta_2 < \frac{\theta}{\eta_1}, \quad \eta_2 \leq \frac{\lambda}{4(1 + \lambda)}, \quad \gamma := \min\left(\frac{\lambda \eta_1}{8}, \frac{\theta - \eta_1 \eta_2 - \eta_1}{2}\right)
\end{align}
leaves us with the estimate 
$$
\sum_{\substack{1 \leq n \leq X \\ n \text{ in case (ii)}}} w_n f(a_n) \leq C_5(\epsilon) M(X)^{\max\left(1 - \frac{(\theta \eta_1 -\eta_2- \eta_1)}{2}+\epsilon, 1 - \frac{\lambda \eta_1}{8}+\epsilon\right)}.
$$
In particular, we can now fix $\epsilon>0$ that depends only on $\theta,\eta_i$ and $\lambda$ so that 
\begin{align}
\label{eConclusionii}
\sum_{\substack{1 \leq n \leq X \\ n \text{ in case (ii)}}} w_n f(a_n) \leq C_5 M(X)^{\max\left(1 - \frac{(\theta \eta_1 -\eta_2- \eta_1)}{4}, 1 - \frac{\lambda \eta_1}{16}\right)}.
\end{align}

\subsection*{Case (iii)}
Since $P^+(c_n) < P^-(b_n) \leq (\log Z)( \log \log Z)$, all prime divisors of $c_n$ are unusually small; this will give a power saving error term. By~\eqref{efWSub} and~\eqref{eGXbound} we obtain for any $\epsilon > 0$
\begin{align*}
\sum_{\substack{1 \leq n \leq X \\ n \text{ in case (iii)}}} w_n f(a_n) &\leq C_6(\epsilon) M(X)^{\epsilon}
\sum_{\substack{1 \leq n \leq X \\ n \text{ in case (iii)}}} w_n \\
&\leq C_6(\epsilon) M(X)^{\epsilon} \sum_{\substack{Z^{1/2} < d \leq Z \\ P^+(d) \leq (\log Z) (\log \log Z)}} 
\sum_{\substack{1 \leq n \leq X \\ d\mid a_n }} w_n,
\end{align*}
where $d = c_n$. Using \eqref{eLevel} and the arguments involving the partitions $\c B \in \c P(m)$ proving~\eqref{eCasei3} we obtain the following upper bound for the sum over $d$:
$$
C_7 M(X) \sum_{\substack{Z^{1/2} < d \leq Z \\ P^+(d) \leq (\log Z)( \log \log Z)}} h(d) + C_7 Z^2 M(X)^{1 - \theta}.
$$ 
We estimate the new sum over $d$ by alluding to \cite[Lemma 2.1]{CKPS} with 
$$
F = h, \ c_0 = \max\left(B, \max_{p \leq B} p \cdot h(p)\right), \ c_1 = \frac{\log B}{\log 2}, \ c_2 = \lambda, \ x = Z, \ z= Z^{1/2},
$$ 
thus obtaining the bound $C_8 Z^{-\phi},$ where $\phi$ is a positive constant that depends on $\lambda$ and $B$. We can assume that $2\eta_1 < \theta$ and then choosing $\epsilon = \eta_1 \phi/2$ and $\gamma = \min(\eta_1 \phi/2, \theta - 2\eta_1) > 0$ we obtain
\begin{align}
\label{eConclusioniii}
\sum_{\substack{1 \leq n \leq X \\ n \text{ in case (iii)}}} w_n f(a_n) \leq C_9 M(X)^{1 - \gamma}.
\end{align}

\subsection*{Case (iv)}
Write $d = c_n$ so that $b_n = a_n/d$ and $d$ are coprime. Since $b_n\md d$ falls in some $\c A\in \c P(d)$ we use~\eqref{efWSub} to get $f(a_n) \leq g(d,\c A) A^{\Omega(b_n)}$. Hence, 
\begin{align}
\label{eOpeniv}
\sum_{\substack{1 \leq n \leq X \\ n \text{ in case (iv)}}} w_n f(a_n) \leq \sum_{Z^{1/2} < d \leq Z} \sum_{\mathcal{A} \in \mathcal{P}(d)} g(d, \mathcal{A}) \Osum_{\substack{1 \leq n \leq X, d \mid a_n \\ (\log Z) (\log \log Z)< P^-(a_n/d) < Z^{\eta_2}}} w_n A^{\Omega(a_n/d)},
\end{align}
where $\Osum$ is subject to the further conditions $\gcd(d, a_n/d) = 1$ and $a_n/d \in \mathcal{A}$. Define the integer $s$ so that $Z^{1/(s + 1)} < P^-(a_n/d) \leq Z^{1/s}$. Letting
$$
s_0 := \left \lfloor \frac{\log Z}{\log(\log Z \log \log Z)} \right \rfloor \leq \frac{\log Z}{\log \log Z}
$$ 
we infer $1 \leq s \leq s_0$ by the definition of case (iv). Further, $\eta_1 \Omega(a_n/d) \leq 3 s \alpha$ owing to 
\[
M(X)^{\frac{\eta_1 \Omega(a_n/d)}{2s}} \leq M(X)^{\frac{\eta_1 \Omega(a_n/d)}{s + 1}} = Z^{\Omega(a_n/d)/(s + 1)} < P^-(a_n/d)^{\Omega(a_n/d)} \leq a_n \leq \alpha M(X)^\alpha.
\]
Therefore,~\eqref{eOpeniv} is at most 
$$ 
\sum_{1 \leq s \leq s_0} A^{3 s \alpha \eta_1^{-1}} \sum_{\substack{Z^{1/2} < d \leq Z \\ P^+(d) < Z^{1/s}}} \sum_{\mathcal{A} \in \mathcal{P}(d)} g(d, \mathcal{A}) \Osum_{\substack{1 \leq n \leq X, d \mid a_n \\ Z^{1/(s + 1)} < P^-(a_n/d) \leq Z^{1/s}}} w_n.
$$ 
The condition $Z^{1/(s + 1)} < P^-(a_n/d)$ will be dealt via~\cite[Lemma 6.3]{iwaniec} with $y = Z$. Set 
$$
P_s := \prod_{\substack{p \in (B, Z^{1/(s + 1)}] \\ p \nmid d}} p.
$$
We obtain 
$$
\Osum_{\substack{1 \leq n \leq X, d \mid a_n 
 \\ Z^{1/(s + 1)} < P^-(a_n/d) \leq Z^{1/s}}} w_n
\leq \sum_{\substack{1 \leq n \leq X, d \mid a_n \\ \gcd(P_s, a_n/d) = 1, a_n/d \in \mathcal{A}}} w_n \leq \sum_{m \mid P_s} \lambda_m^+ \sum_{\substack{1 \leq n \leq X, dm \mid a_n \\ a_n/d \in \mathcal{A}}} w_n.
$$ 
Arguing as in the analogous step in case (i) we obtain the upper bound 
$$ 
C_{10} \bigg(h(d, \mathcal{A}) M(X) \prod_{\substack{B < p < Z^{1/(s + 1)} \\ p \nmid d}} (1 - h(p)) + Z^2 M(X)^{1 - \theta/2}\bigg).
$$ 
Now \eqref{eFSieveAss} allows us to extend the product over $p$ all the way up to $M(X)$ at the expense of an error of size $C_{11} (s + 1)^\kappa$. This shows that the main term in the last equation contributes 
\begin{equation}
\label{eSmoothnessSaving}
\ll \prod_{B < p \leq M(X)} (1 - h(p)) \sum_{1 \leq s \leq s_0} A^{3 s \alpha \eta_1^{-1}} (s + 1)^\kappa \sum_{\substack{Z^{1/2} < d \leq Z \\ P^+(d) < Z^{1/s}}} h(d) H(d)\prod_{B<p \mid d} (1 - h(p))^{-1},
\end{equation} 
where $H$ is as in Theorem~\ref{tNT}. For sufficiently large $X$, we apply \cite[Lemma 2.6]{CKPS} with $\Upsilon := Z^{1/2}$, $\Psi := Z^{1/s}$, $F := h$, $G := H$ and $\varpi = 6 \alpha \eta_1^{-1} \log(4A)$. Note that conditions on $F, G$ are satisfied thanks to \eqref{ehOnPrimes}, \eqref{ehDecay} and \eqref{eHSubm}. We then take $\beta_0 := 6 \alpha \eta_1^{-1} \log(4A)$ to obtain the estimate
$$
\sum_{\substack{Z^{1/2} < d \leq Z \\ P^+(d) < Z^{1/s}}} h(d) H(d) \prod_{B<p \mid d} (1 - h(p))^{-1} \ll (4A)^{-3 s \alpha \eta_1^{-1}} \sum_{d \leq Z} h(d) H(d) \prod_{B<p \mid d} (1 - h(p))^{-1}.
$$ 
This makes \eqref{eSmoothnessSaving} be 
$$
\ll \prod_{B < p \leq M(X)} (1 - h(p)) \sum_{1 \leq s \leq s_0} 4^{-3 s \alpha \eta_1^{-1}} (s + 1)^\kappa \sum_{d \leq Z} h(d) H(d) \prod_{B<p \mid d} (1 - h(p))^{-1}.
$$
The sum over $s$ converges, thus, by \cite[Lemma 2.7]{CKPS} we get 
\begin{equation}
\label{eConclusioniv}
\sum_{\substack{1 \leq n \leq X \\ \text{case (iv)}}} w_n f(a_n) \leq C_{12} \prod_{B < p \leq M(X)} (1 - h(p)) \sum_{d \leq Z} h(d) H(d) + C_{12} Z^4 M(X)^{1 - \theta/2}.
\end{equation}

\subsection*{Proof of Theorem~\ref{tNT} }
In case (i) and (iv) we assumed $Z^2 \leq M(X)^\theta$, while, in case (iii) we assumed \eqref{eParChoices}. Pick $\eta_1 > 0$ sufficiently small and then pick $\eta_2 > 0$ sufficiently small in terms of $\eta_1$ and the other parameters. Putting together \eqref{eConclusioni}, \eqref{eConclusionii}, \eqref{eConclusioniii}, \eqref{eConclusioniv} and absorbing the power savings into the main term concludes the proof. 
\end{proof}

\subsection{Reducing to square-frees} 
It is useful to work with simpler sums than the one over $d$ in Theorem~\ref{tNT}. 
We give a list of assumptions under which such a simplification is possible:

\begin{lemma}
\label{tSqfReduction}
Let $\kappa \geq 1$ be a real number and let $f^\ast: \Z_{\geq 1} \rightarrow [0, \infty)$ be such that
\begin{itemize}
\item $f^\ast(ab) \leq f^\ast(a) \kappa^{\omega(b)}$ for all coprime $a, b \geq 1$,
\item $f^\ast(as^2) \leq f^\ast(a)$ for all $a, s \geq 1$.
\end{itemize}
Fix constants $B > 10, c > 0$ and assume that $h: \Z_{\geq 1} \rightarrow [0, \infty)$ is multiplicative and satisfies 
\begin{itemize}
\item $h(p^e) \leq h(p^2) \leq Bp^{-2}$ for all $e \geq 2$ and primes $p$,
\item $h(p^e) \leq B p^{-ce}$ for all $e \geq 1$ and primes $p$.
\end{itemize} 
Then for all $X\geq 2$ we have 
\[
\sum_{a \leq X} f^\ast(a) h(a) \ll \sum_{\substack{1 \leq a \leq X \\ a \ \mathrm{square-free}}} f^\ast(a) h(a).
\]
\end{lemma}

\begin{proof}
Each $a \in \Z_{\geq 1}$ factors uniquely as $ \alpha^2 \beta \gamma$, where $\mu(\beta \gamma)^2 = 1$, $ \beta \mid \alpha$ and $\gcd(\alpha \beta, \gamma) = 1$. Then
$$
f^\ast(\alpha^2 \beta \gamma) h(\alpha^2 \beta \gamma) \leq \kappa^{\omega(\beta)} f^\ast(\gamma) h(\alpha^2 \beta) h(\gamma),
$$ 
thus, the sum in the lemma is at most 
$$
\sum_{\substack{\alpha^2 \beta \leq X \\ \beta \mid \alpha}} \mu(\beta)^2 h(\alpha^2 \beta) \kappa^{\omega(\beta)} 
\hspace{-0,2cm}
\sum_{\gamma \leq X/(\alpha^2 \beta)}
\hspace{-0,2cm}
\mu(\gamma)^2 f^\ast(\gamma) h(\gamma)
\leq \sum_{\gamma \leq X} \mu(\gamma)^2 f^\ast(\gamma) h(\gamma)
\sum_{\substack{\alpha, \beta \geq 1 \\ \beta \mid \alpha}} \mu(\beta)^2 h(\alpha^2 \beta) \kappa^{\omega(\beta)},
$$ 
where we used the non-negativity of the values of $h$ and $f^*$. The new sum over $\alpha, \beta$ equals 
$$
\prod_p \bigg(1 + \sum_{e = 1}^\infty \bigg(\kappa h(p^{2e + 1}) + h(p^{2e})\bigg)\bigg).
$$
If $p \leq 2^{1/c}$ the sum converges by $h(p^e) \leq B p^{-ce}$. For $p > 2^{1/c}$ and $E = 1 + [2/c]$, the sum is 
$$
\leq \frac{(1 + \kappa) BE}{p^2} + (1 + \kappa) B \sum_{e > E} p^{-ce} \leq \frac{(1 + \kappa) BE}{p^2} + \frac{2 (1 + \kappa) B}{p^{c E}} \leq \frac{(1 + \kappa) B (E + 2)}{p^2},
$$ 
hence, the product over $p$ converges.
\end{proof}

\section{Weighted moments}
\label{sMoments}
We will now handle sums rather similar to the output of Section \ref{sNT}. Let $g(m, n)$ be the R\'edei majorant from Definition~\ref{def:firstdefn} or equation~\eqref{def:stackbundles}. These are of the type described in Definition \ref{dClasses} but they have extra structure coming from quadratic residues. To be precise, let $q_1 < \dots < q_r$ for the odd prime divisors of $m$ so that $g(m, n)$ depends only on the class $\bm{\epsilon}$ of $n$ in
$$
\prod_{i = 1}^r \frac{(\Z/q_i\Z)^\ast}{(\Z/q_i\Z)^{\ast 2}},
$$
which allows us to introduce the quantity $g(m, \bm{\epsilon})$ for every $\bm{\epsilon} \in \mathbb{F}_2^r$. Following \S
\ref{sNT}, we are inspired to calculate the weighted moment
\begin{align}
\label{ef1def}
\sum_{\substack{1 \leq m \leq X \\ m \text{ odd}}} \mu(m)^2 f^\ast(m) h(m), \quad \quad f^\ast(m) = \frac{1}{2^r} \sum_{\bm{\epsilon} = (\epsilon_i)_{1 \leq i \leq r}} g(m, \bm{\epsilon}).
\end{align} 
One can replace $h(m)$ by $F(m) = m h(m)$ via partial summation. The following class of functions is appropriate for character techniques:

\begin{definition}
\label{dSuitableh}
We say that a non-negative multiplicative function $F$ is appropri\'ee if 
\begin{enumerate} 
\item[(i)] $F(n) \leq \tau(n)^L$ for all $n\geq 1$, 
\item[(ii)] there exists $\alpha>0$ and $C(\alpha)>0$ such that for $X\geq 2 $ we have 
\begin{equation}
\label{eq:Calpha} \prod_{p\leq X} \bigg(1+\frac{F(p)}{p}\bigg) \geq C(\alpha) (\log X)^{\alpha},
\end{equation}
\item[(iii)] there exists a finite exceptional set $E\subseteq \Z_{\geq 1}$ such that for all fixed real numbers $A > 1$ there exists a constant $C = C(A) > 0$ 
for which 
\begin{align}
\label{eSW}
\left|\sum_{p \leq X} F(p) \chi(p)\right| \leq \frac{CX}{(\log X)^A}
\end{align} 
whenever $X \geq 2$ and $\chi$ is a non-principal, quadratic, primitive Dirichlet character of conductor $q \not \in E$ bounded by $(\log X)^A$.
\end{enumerate}
\end{definition}

\noindent Condition $(i)$ will come up in the large sieve amongst other things. The assumptions in condition $(ii)$ will be important when trivially bounding the contribution from too many small variables. Condition $(iii)$ is a typical Siegel--Walfisz type condition. We need to allow for the exceptional moduli in some of our applications of algebraic nature.

Our next theorem, which is the main theorem of this section, achieves a good control on the weighted moments provided that $F$ is appropri\'ee.

\begin{theorem}
\label{tCharacter}
Let $k \geq 1$ be an integer, let $g(m, n)$ be the R\'edei majorant satisfying either Definition~\ref{def:firstdefn} or~\eqref{def:stackbundles}, and let
$$
f_k^\ast(m) := \left(\frac{1}{2^r} \sum_{\bm{\epsilon} = (\epsilon_i)_{1 \leq i \leq r}} g(m, \bm{\epsilon})\right)^k.
$$ 
Assume that $F$ is appropri\'ee. Then we have
$$
\sum_{1 \leq m \leq X} \mu(2m)^2 f_k^\ast(m) F(m) \ll \frac{X}{\log X} \prod_{p \leq X} \left(1 + \frac{F(p)}{p}\right),
$$ 
where the implied constant depends on $F$ and $k$.
\end{theorem}

Before we embark on the proof of Theorem \ref{tCharacter}, we will state some well-known oscillation results of use to us.

\subsection{Oscillation results}
Various double oscillation results in the literature are available \cite{San, Wil}, starting from the pioneering work of Heath-Brown \cite{HB}. We will use the following variation.

\begin{lemma}
\label{lLargeSieve}
Let $k \geq 1$ be an integer. Then there exists a constant $C > 0$ depending only on $k$ such that the following holds. Let $\alpha_m, \beta_n$ be sequences of complex numbers supported on odd, square-free numbers satisfying $|\alpha_m| \leq \tau(m)^k, |\beta_n| \leq \tau(n)^k$.
Then for all $X, Y \geq 2$ we have 
\[
\left|\sum_{1 \leq m \leq X} \sum_{1 \leq n \leq Y} \alpha_m \beta_n \left(\frac{m}{n}\right)\right| \leq C XY(X^{-1/6} + Y^{-1/6}) (\log XY)^C.
\]
\end{lemma}

\begin{proof}Note that Koymans--Rome \cite[Proposition 4.3]{KR}, or alternatively \cite[Lemma 2]{MR2675875},
requires that the coefficients $\alpha_m$, $\beta_n$ are bounded by $1$ in absolute value. However, the proof goes through with straightforward modifications in the more general setting where $\alpha_m$, $\beta_n$ are divisor bounded.
\end{proof}

The next result is a version of the work in~\cite{Wil}.

\begin{corollary}
\label{cLargeSieve}
Let $s \geq 1, r \geq 2$ be integers. Then there exists a constant $C > 0$ depending only on $s$ and $r$ such that the following holds. Let $\alpha, \beta: \Z_{\geq 1}^{r - 1} \to \mathbb C$ be supported on odd, square-free numbers satisfying $|\alpha(\b n)|, |\beta(\b n)| \leq \tau(n_1)^s\cdots \tau(n_{r - 1})^s$. Then for all $X, z \geq 2$ we have 
\[
\left|\sum_{\substack{\b m \in \Z_{\geq 1 }^r, m_1\cdots m_r \leq X \\ m_1,m_2 > z}} 
\left(\frac{m_1}{m_2}\right) \alpha(m_1, m_3, \ldots, m_r) \beta(m_2, m_3, \ldots, m_r)\right| \leq \frac{CX (\log X)^C}{z^{1/20}}.
\]
\end{corollary}

\begin{proof} 
We first deal with the case $r=2$ and we will at the end deal with general $r$. Set $A = 1 + z^{-1/20}$ and define $I_i = (zA^i, zA^{i + 1}]$ for an integer $i\geq 0$. We next consider all integers $i,j\geq 0$ such that the box $I_i\times I_j$ is contained inside the hyperbola $mn\leq X$. 
For each such box we use Lemma \ref{lLargeSieve} to obtain an error term $\ll X z^{-1/6} (\log X)^C$. To multiply this error term by the total number of boxes note that 
we need $\ll z^{1/10} (\log X)^2$ boxes
 to cover $[1,X]^2$ , therefore, the resulting error term is $\ll X z^{1/10-1/6} (\log X)^{C+2}$.

The $(m, n)$ that are left over satisfy 
\begin{align}
\label{eBadRange}
X - \frac{cX}{z^{1/20}} \leq mn \leq X
\end{align}
for some absolute constant $c > 0$. Indeed, if $I_i\times I_j$ intersects the interior and the exterior of the hyperbola then $z^2 A^{i + j} \leq X \leq z^2 A^{i+j+2}$, from which one can easily deduce that the remaining $(m, n) \in I_i\times I_j$ satisfy $mn \geq X(1-z^{-1/20})^{-2}$ that proves~\eqref{eBadRange}. Using the divisor function bounds on $\alpha, \beta$ and setting $t = mn$ the left over region makes a contribution that is 
$$\ll \sum_{X-cXz^{-1/20}\leq t\leq X} \tau(t) \sum_{mn=t} \tau(m)^s\tau(n)^s\leq 
\sum_{X-cXz^{-1/20}\leq t\leq X} \tau(t)^{2+2s}\ll \frac{X}{z^{1/20}} (\log X)^{4^{s+1}}
$$ by Shiu's work \cite[Theorem 1]{Shiu}. We may freely assume that $z\leq X$ since otherwise the theorem is trivial, hence, the assumptions in Shiu's theorem are met. This concludes the proof when $r = 2$.

We now prove the general case with any $r \geq 2$. Define 
$$
\widetilde {\alpha}(m_1,m_3,\ldots,m_r)=\frac{ \alpha(m_1,m_3,\ldots,m_r)}{\tau(m_3)^s\cdots \tau(m_r)^s}
\ \textrm{ and } \ 
\widetilde {\beta}(m_2,m_3,\ldots, m_r)=\frac{ \beta(m_2,m_3,\ldots, m_r)}{\tau(m_2)^s\cdots \tau(m_r)^s}.
$$ 
The triangle inequality yields the bound 
$$
\sum_{\substack{\b m \in \Z_{\geq 1}^{r - 2} \\ m_3 \cdots m_r \leq X}} \tau(m_3)^{2s} \cdots \tau(m_r)^{2s} 
\left|\sum_{\substack{m_1,m_2 \in \Z_{>z } \\ m_1 m_2 \leq X/(m_3\cdots m_r)}} 
\left(\frac{m_1}{m_2}\right) \widetilde \alpha(m_1, m_3, \ldots,m_r) \widetilde \beta(m_2, m_3,\ldots, m_r)\right|.
$$ 
By our assumptions we have $|\widetilde {\alpha}(m_1,m_3,\ldots,m_r)|\leq \tau(m_1)^s$ and $|\widetilde {\beta}(m_2,m_3,\ldots, m_r)|\leq \tau(m_2)^s$, hence, we can use the known special case $r = 2$ for the inner sum over $m_1, m_2$.
We get 
$$
\ll \frac{X(\log X)^C}{z^{1/20}} \prod_{i = 3}^r \sum_{\substack{m_i \leq X}} \frac{\tau(m_i)^{2s} }{m_i}
\ll \frac{X(\log X)^{C'}}{z^{1/20}},
$$ 
where $C' = C + (r - 2)4^s$. 
\end{proof}

Our next theorem will be used to convert information on partial sums over primes to partial sums over all integer values. We employ this result from the work of Granville--Koukoulopoulos \cite{GK} in the version stated by Koukoulopoulos {\cite[Theorem 13.2]{Kou}}. The key feature that is useful to us is the explicit dependence of the implied constant on the multiplicative function. 

\begin{theorem}[Beyond \texttt{LSD}]
\label{tKou}
Let $Q \geq 2$ be a parameter and $f$ be a multiplicative function with
\begin{align}
\label{eOnPrimes}
\sum_{p \leq x} f(p) \log p = O_A\left( \frac{x}{(\log x)^A} \right) \quad (x \geq Q)
\end{align}
for all $A > 0$. Also assume that $\vert f(n) \vert \leq \tau_k(n)$ for some positive real number $k$. Fix $\epsilon > 0$ and $J \in \Z_{\geq 1}$. Then
for all $x \geq e^{(\log Q)^{1 + \epsilon}}$ we have 
$$
\sum_{n \leq x} f(n) = O\left(\frac{x (\log Q)^{2k + J - 1}}{(\log x)^{J + 1 - \textup{Re}(\alpha)}}\right).
$$
The implied constant depends at most on $k$, $J$, $\epsilon$ and the implied constant in \eqref{eOnPrimes} for $A$ large enough in terms of $k$, $J$ and $\epsilon$ only.
\end{theorem}

\subsection{First moment} 
Our discussion will naturally split in two cases corresponding to Theorems~\ref{tFibration} and~\ref{t12}.

\subsubsection{Number field setting}
We start by rewriting our sum in case $k = 1$ and $g(m, n)$ is the R\'edei majorant from Definition~\ref{def:firstdefn}. Recall that $q_1 < \dots < q_r$ denote the odd prime divisors of $m$ and that $g(m, n)$ depends only on the class $\bm{\epsilon}$ of $n$ in
$$
\prod_{i = 1}^r \frac{(\Z/q_i\Z)^\ast}{(\Z/q_i\Z)^{\ast 2}}.
$$
Also recall the definition of $f_1^\ast(m)$ in equation \eqref{ef1def} and the definition of $g(m, \bm{\epsilon})$ for $\bm{\epsilon} \in \mathbb{F}_2^r$. Given $\bm{\epsilon}$, an odd square-free integer $m$ and a prime $p$ dividing $m$, we define $t(p, m, \bm{\epsilon})$ to be $(-1)^{\epsilon_i}$, where $i$ is the unique integer such that $p$ is the $i$-th smallest prime divisor of $m$, i.e. $p = q_i$.

Recalling Definition~\ref{def:firstdefn} we see that the first weighted moment 
$$
\sum_{1 \leq m \leq X} \mu(2m)^2 f_1^\ast(m) F(m)
$$
becomes
$$
\sum_{r \geq 0} \frac{1}{2^r} \sum_{\bm{\epsilon} = (\epsilon_i)_{1 \leq i \leq r}} \sum_{\substack{m \leq X \\ \omega(m) = r}} 
F(m) \mu(2m)^2 \sum_{d \mid m} \frac{1}{2^r} \prod_{p \mid d} \left(1 + t(p, m, \bm{\epsilon}) \left(\frac{m/d}{p}\right)\right) \prod_{p \mid \frac{m}{d}} \left(1 + \left(\frac{d}{p}\right)\right),
$$ 
where $t(p, m, \bm{\epsilon})$ comes from $\chi_\alpha(\mathrm{Frob}_{q_i})$. From linear algebra, we get the identity
$$
\sum_{\bm{\epsilon} = (\epsilon_i)_{1 \leq i \leq r}} \prod_{p \mid d} \left(1 + t(p, m, \bm{\epsilon}) \left(\frac{m/d}{p}\right)\right) = 2^r
$$
by viewing the product as detecting solutions of $\omega(d)$ linearly independent equations in the variables $\epsilon_i$ with 
each solution being counted with weight $2^{\omega(d)}$. Thus, summing over $\bm{\epsilon}$ gives
\[
\sum_{m \leq X} \mu(2m)^2 f_1^\ast(m) F(m) =
\sum_{m \leq X} \frac{\mu(2m)^2 F(m)}{2^{\omega(m)}} \sum_{d \mid m} \prod_{p \mid \frac{m}{d}} \left(1 + \left(\frac{d}{p}\right)\right) = \sum_{def \leq X} \frac{\mu(2def)^2 F(def)}{2^{\omega(def)}} \left(\frac{d}{e}\right).
\]

\subsubsection{Elliptic curve setting}
Let us now suppose that $g_\mathbf{r}(m, n)$ is the R\'edei majorant from ~\eqref{def:stackbundles}. As before we let $m = q_1 \cdots q_r$ with $q_1 < 
\ldots < q_r$ all coprime to $\Omega$. Then the vector space $W'$ from \S\ref{ssSelmer} 
consists of pairs $(x_1, x_2)$, which are two positive integers dividing $m$. Alternatively, we may think of $W'$ as quadruplets $(D_1, D_2, D_3, D_4)$ with
$
m = D_1 D_2 D_3 D_4
$
and $D_1, D_2, D_3, D_4$ positive and coprime via the change of variables $x_1 = D_1D_2$ and $x_2 = D_1D_3$. Let us now detect 
when $(D_1, D_2, D_3, D_4)$ lies in the kernel of $M'_\mathbf{r}(m, \bm{\epsilon})$. 
This operation will be similar to \cite[Lemma 3]{HBCongruent}, or may alternatively be derived by 
studying \eqref{ePhiv} and using properties of local Hilbert symbols. Let 
\begin{align*}
F_1 &= \prod_{p \mid D_1} \frac{1}{4} \left(1 + t(p, m, \bm{\epsilon}) \left(\frac{\delta_{31} D_3 D_4}{p}\right) + t(p, m, \bm{\epsilon}) \left(\frac{\delta_{32} D_2 D_4}{p}\right) + \left(\frac{\delta_{31} \delta_{32} D_2 D_3}{p}\right)\right), \\
F_2 &= \prod_{p \mid D_2} \frac{1}{4} \left(1 + t(p, m, \bm{\epsilon}) \left(\frac{\delta_{21} D_3 D_4}{p}\right) + \left(\frac{\delta_{21} \delta_{23} D_1 D_3}{p}\right) + t(p, m, \bm{\epsilon}) \left(\frac{\delta_{23} D_1 D_4}{p}\right)\right), \\
F_3 &= \prod_{p \mid D_3} \frac{1}{4} \left(1 + \left(\frac{\delta_{12} \delta_{13} D_1 D_2}{p}\right) + t(p, m, \bm{\epsilon}) \left(\frac{\delta_{12} D_2 D_4}{p}\right) + t(p, m, \bm{\epsilon}) \left(\frac{\delta_{13} D_1 D_4}{p}\right)\right), \\
F_4 &= \prod_{p \mid D_4} \frac{1}{4} \left(1 + \left(\frac{D_1 D_2}{p}\right) + \left(\frac{D_1 D_3}{p}\right) + \left(\frac{D_2 D_3}{p}\right)\right).
\end{align*} Then the detector function of $(D_1, D_2, D_3, D_4)$ lying in the kernel of $M'_\mathbf{r}(m, \bm{\epsilon})$ is exactly $F_1 F_2 F_3 F_4$. We may for instance expand $F_1$ as
$$
F_1 = \frac{1}{4^{\omega(D_1)}} \sum_{D_1 = D_{10} D_{12} D_{13} D_{14}} \left(\frac{\delta_{31} D_3 D_4}{D_{12}}\right) \left(\frac{\delta_{32} D_2 D_4}{D_{13}}\right) \left(\frac{\delta_{31} \delta_{32} D_2 D_3}{D_{14}}\right) \times \prod_{p \mid D_{12} D_{13}} t(p, m, \bm{\epsilon}),
$$
where the sum is over all factorizations $D_1 = D_{10} D_{12} D_{13} D_{14}$. Doing this also for $F_2, F_3, F_4$ yields
$$
f_1^\ast(m) = \frac{1}{8^{\omega(m)}} \sum_{\bm{\epsilon} = (\epsilon_i)_{1 \leq i \leq \omega(m)}} \sum_{m = D_1 D_2 D_3 D_4} 
\lambda(\mathbf{D}) \prod_{1 \leq i \leq 4} \prod_{\substack{0 \leq j \leq 4 \\ i \neq j}} \prod_{k \neq i, j} \prod_{l \neq k} \left(\frac{D_{kl}}{D_{ij}}\right),
$$ 
where the $D_i$ take the shape 
\begin{align*}
&D_1 = D_{10} D_{12} D_{13} D_{14}, \ \ D_2 = D_{20} D_{21} D_{23} D_{24} \\ 
& D_3 = D_{30} D_{31} D_{32} D_{34}, \ \ D_4 = D_{40} D_{41} D_{42} D_{43},
\end{align*} 
where $\mathbf{D}$ is the vector of $D_{ij}$ and
\begin{align*}
\lambda(\mathbf{D}) &:= \lambda'(\mathbf{D}) \prod_{p \mid D_{12} D_{13} D_{21} D_{23} D_{31} D_{32}} t(p, m, \bm{\epsilon}), \\
\lambda'(\mathbf{D}) &:= \left(\frac{\delta_{31}}{D_{12} D_{14}}\right) \left(\frac{\delta_{32}}{D_{13} D_{14}}\right) 
\left(\frac{\delta_{21}}{D_{12} D_{24}}\right) \left(\frac{\delta_{23}}{D_{23} D_{24}}\right) 
\left(\frac{\delta_{12}}{D_{13} D_{34}}\right) \left(\frac{\delta_{13}}{D_{23} D_{34}}\right).
\end{align*}
Indeed, when expanding all Legendre symbols in $F_1F_2F_3F_4$, one notices that the term $(D_{kl}/D_{ij})$ appears exactly when firstly $k \neq l$ and $i \neq j$ (so $D_{kl}$ and $D_{ij}$ are defined), and additionally $k \neq i, j$. 

If $D_{12} D_{13} D_{21} D_{23} D_{31} D_{32} > 1$, we average over $\bm{\epsilon}$ to show that the sum vanishes. However, to keep the parallel between our work and \cite{HBCongruent} as much as possible, we retain the variables $D_{12}, D_{13}, D_{21}, D_{23}, D_{31}, D_{32}$. Therefore we may rewrite $f_1^\ast(m)$ as
$$
f_1^\ast(m) = \frac{1}{4^{\omega(m)}} \sum_{m = D_1 D_2 D_3 D_4} 
\lambda'(\mathbf{D}) \prod_{1 \leq i \leq 4} \prod_{\substack{0 \leq j \leq 4 \\ i \neq j}} \prod_{k \neq i, j} \prod_{l \neq k} \left(\frac{D_{kl}}{D_{ij}}\right),
$$ where the sum over $m$ is subject to $D_1 = D_{10} D_{12} D_{13} D_{14}$, $D_2 = D_{20} D_{21} D_{23} D_{24}$ and 
$$D_3 = D_{30} D_{31} D_{32} D_{34}, \ \ D_4 = D_{40} D_{41} D_{42} D_{43}, \ \ D_{12} D_{13} D_{21} D_{23} D_{31} D_{32} = 1.$$
The resulting moment is $$
\sum_{\substack{\mathbf{D} \\ \prod_{i, j} D_{ij} \leq X}} \frac{\mu(\Omega \prod_{i, j} D_{ij})^2}{4^{\omega(\prod_{i, j} D_{ij})}} \times F\left(\prod_{i, j} D_{i, j}\right) \times \lambda'(\mathbf{D}) \times \prod_{1 \leq i \leq 4} \prod_{\substack{0 \leq j \leq 4 \\ i \neq j}} \prod_{k \neq i, j} \prod_{l \neq k} \left(\frac{D_{kl}}{D_{ij}}\right).
$$
In order to prepare for the computation of the higher moments
we rewrite the above in a compact notation. The variables $D_{ij}$ will henceforth be indexed by $\mathbb{F}_2^4$ according to
\begin{alignat*}{4}
&D_{10} = D_{0001}, \quad &&D_{12} = D_{1011}, \quad &&D_{13} = D_{1001}, \quad &&D_{14} = D_{0011} \\
&D_{20} = D_{0100}, \quad &&D_{21} = D_{1110}, \quad &&D_{23} = D_{0110}, \quad &&D_{24} = D_{1100} \\
&D_{30} = D_{0101}, \quad &&D_{31} = D_{1101}, \quad &&D_{32} = D_{0111}, \quad &&D_{34} = D_{1111} \\
&D_{40} = D_{0000}, \quad &&D_{41} = D_{0010}, \quad &&D_{42} = D_{1000}, \quad &&D_{43} = D_{1010}.
\end{alignat*}
The purpose of this change of variables is that now the Jacobi symbol $(D_{kl}/D_{ij})$ occurs in the new variables $\mathbf{u}, \mathbf{v} \in \mathbb{F}_2^4$ if and only if $\psi(\mathbf{u}, \mathbf{v}) = 1$, where $\psi$ is the bilinear form
$$
\psi(\mathbf{u}, \mathbf{v}) = v_1(u_4 + v_4) + v_3(u_2 + v_2),
$$
see \cite[p. 338]{HBCongruent}. Then our weighted moment becomes
$$
\sum_{\substack{(D_\mathbf{u})_{\mathbf{u} \in \mathbb{F}_2^4} \\ \prod_\mathbf{u} D_\mathbf{u} \leq X}} \frac{\mu(\Omega \prod_\mathbf{u} D_\mathbf{u})^2}{4^{\omega(\prod_\mathbf{u} D_\mathbf{u})}} \times F\left(\prod_\mathbf{u} D_\mathbf{u}\right) \times \lambda'((D_\mathbf{u})_{\mathbf{u} \in \mathbb{F}_2^4}) \times \prod_{\mathbf{u}, \mathbf{v} \in \mathbb{F}_2^4} \left(\frac{D_\mathbf{u}}{D_\mathbf{v}}\right)^{\psi(\mathbf{u}, \mathbf{v})}
$$
in the new variables.

\subsection{Higher moments}
We distinguish cases between class groups and Selmer groups.

\subsubsection{Higher class moments}
\label{ssHigherClass}
The $k$-th moment equals
$$
\sum_{r \geq 0} \frac{1}{2^r} \sum_{\bm{\epsilon} = (\epsilon_i)_{1 \leq i \leq r}} \sum_{\substack{m \leq X \\ \omega(m) = r}} \mu(2m)^2 F(m) \left(\sum_{d \mid m} \frac{1}{2^r} \prod_{p \mid d} \left(1 + t(p, m, \bm{\epsilon}) \left(\frac{m/d}{p}\right)\right) \prod_{p \mid \frac{m}{d}} \left(1 + \left(\frac{d}{p}\right)\right)\right)^k.
$$
We rewrite this as
$$
\sum_{r \geq 0} \frac{1}{2^{r(k + 1)}} \hspace{-0.35cm} \sum_{\bm{\epsilon} = (\epsilon_i)_{1 \leq i \leq r}} \hspace{-0.2cm} \sum_{\substack{m \leq X \\ \omega(m) = r}} \mu(2m)^2 F(m) \hspace{-0.35cm} \sum_{d_1, \dots, d_k \mid m} \prod_{i = 1}^k \left(\prod_{p \mid d_i} \left(1 + t(p, m, \bm{\epsilon}) \left(\frac{m/d_i}{p}\right)\right) \prod_{p \mid \frac{m}{d_i}} \left(1 + \left(\frac{d_i}{p}\right)\right)\right).
$$
Setting $t(e, m, \bm{\epsilon}) := \prod_{p \mid e} t(p, m, \bm{\epsilon})$, we expand the products over $p$ to get
$$
\sum_{r \geq 0} \frac{1}{2^{r(k + 1)}} \sum_{\bm{\epsilon} = (\epsilon_i)_{1 \leq i \leq r}} \sum_{\substack{m \leq X \\ \omega(m) = r}} \mu(2m)^2 F(m) \sum_{d_1, \dots, d_k \mid m} \prod_{i = 1}^k 
\sum_{e_i \mid d_i} \sum_{f_i \mid \frac{m}{d_i}} t(e_i, m, \bm{\epsilon}) \left(\frac{m/d_i}{e_i}\right) \left(\frac{d_i}{f_i}\right).$$
We continue by also expanding the product over $i$ as follows:
$$
\sum_{r \geq 0} \frac{1}{2^{r(k + 1)}} \sum_{\bm{\epsilon}} \sum_{\substack{m \leq X \\ \omega(m) = r}} \mu(2m)^2 F(m) \hspace{-0.45cm} \sum_{\substack{d_{1, 1} d_{1, 2} d_{1, 3} d_{1, 4} = m \\ \vdots \\ d_{k, 1} d_{k, 2} d_{k, 3} d_{k, 4} = m}} \prod_{i = 1}^k \left(t(d_{i, 1}, m, \bm{\epsilon}) \left(\frac{d_{i, 3} d_{i, 4}}{d_{i, 1}}\right) \left(\frac{d_{i, 1} d_{i, 2}}{d_{i, 3}}\right)\right),
$$
where $e_i = d_{i, 1}, d_i/e_i = d_{i, 2}, f_i = d_{i, 3}, m/(d_if_i) = d_{i, 4}$. Biject $\mathbb{F}_2^2$ with $\{1, 2, 3, 4\}$ by sending $(0, 0)$ to $1$, $(0, 1)$ to $2$, $(1, 1)$ to $3$ and $(1, 0)$ to $4$ and write $B$ for the bijection map. For each $\mathbf{u} \in \mathbb{F}_2^{2k}$, we introduce the new variable $D_\mathbf{u}$ defined through
\[
D_\mathbf{u} := \gcd(d_{1, B(\pi_1(\mathbf{u}))}, \ldots, d_{k, B(\pi_k(\mathbf{u}))}),
\]
where $\pi_i(\mathbf{u})$ is the projection map on the $i$-th copy of $\mathbb{F}_2^2$ by viewing $\mathbb{F}_2^{2k} \cong (\mathbb{F}_2^2)^k$. For each integer $i \in \{1, \dots, 4\}$ and each $\mathbf{u} \in \mathbb{F}_2^{2k}$, we define the operator $S_i(\mathbf{u}) \in \mathbb{F}_2$ to be the parity of the number of indices $j$ such that $B(\pi_j(\mathbf{u})) = i$. We also define the forms
\[
\phi_i(\mathbf{u}, \mathbf{v}) := 
\begin{cases}
1 &\text{if } (B(\pi_i(\mathbf{u})), B(\pi_i(\mathbf{v}))) \in \{(1, 4), (3, 2)\} \\
0 &\text{otherwise,}
\end{cases}
\]
and $\phi(\mathbf{u}, \mathbf{v}) := \sum_{i = 1}^k \phi_i(\mathbf{u}, \mathbf{v})$. We also fix invertible congruence classes $\bm{a} = (a_\mathbf{u})_{\mathbf{u} \in \F_2^{2k}}$ modulo $4$ for each $D_\mathbf{u}$. Applying the triangle inequality and changing variables yields
$$
\sum_{\bm{a}} \left|\sum_{r \geq 0} \frac{1}{2^{r(k + 1)}} \sum_{\bm{\epsilon} = (\epsilon_i)_{1 \leq i \leq r}} \sum_{(D_\mathbf{u}) \in \mathcal{D}(X, k, r, \bm{a})} \prod_{\mathbf{u} \in \mathbb{F}_2^{2k}} F(D_\mathbf{u}) t\left(D_\mathbf{u}, \prod_{\mathbf{v} \in \mathbb{F}_2^{2k}} D_\mathbf{v}, \bm{\epsilon}\right)^{S_1(\mathbf{u})} \prod_{\mathbf{u}, \mathbf{v}} \left(\frac{D_\mathbf{u}}{D_\mathbf{v}}\right)^{\phi(\mathbf{u}, \mathbf{v})}\right|,
$$
where $\mathcal{D}(X, k, r, \bm{a})$ is the set of $4^k$-tuples of odd, square-free, positive and coprime integers $(D_\mathbf{u})_\mathbf{u}$, indexed by $\mathbf{u} \in \mathbb{F}_2^{2k}$, satisfying
\[
\prod_{\mathbf{u} \in \mathbb{F}_2^{2k}} D_\mathbf{u} \leq X, \quad D_\mathbf{u} \equiv a_\mathbf{u} \bmod 4, \quad \omega\left(\prod_{\mathbf{u} \in \mathbb{F}_2^{2k}} D_\mathbf{u}\right) = r.
\]
From now on we shall treat $\bm{a}$ as fixed and concentrate on the inner sum. Our aim at this stage is to utilize the averaging over $\bm{\epsilon}$. To achieve this, we pull out the remaining terms in the sum to get
$$
\sum_{r \geq 0} \frac{1}{2^{r(k + 1)}} \sum_{(D_\mathbf{u}) \in \mathcal{D}(X, k, r, \bm{a})} \left(\prod_{\mathbf{u}, \mathbf{v}} \left(\frac{D_\mathbf{u}}{D_\mathbf{v}}\right)^{\phi(\mathbf{u}, \mathbf{v})} \times \left(\sum_{\bm{\epsilon} = (\epsilon_i)_{1 \leq i \leq r}} \prod_{\mathbf{u} \in \mathbb{F}_2^{2k}} F(D_\mathbf{u}) t\left(D_\mathbf{u}, \prod_{\mathbf{v} \in \mathbb{F}_2^{2k}} D_\mathbf{v}, \bm{\epsilon}\right)^{S_1(\mathbf{u})}\right)\right).
$$
We note that the application $\bm{\epsilon} \mapsto \prod_{\mathbf{u} \in \mathbb{F}_2^{2k}} t\left(D_\mathbf{u}, \prod_{\mathbf{v} \in \mathbb{F}_2^{2k}} D_\mathbf{v}, \bm{\epsilon}\right)^{S_1(\mathbf{u})} $ is a homomorphism, and it is trivial if and only if $D_\mathbf{u} = 1$ for all $\mathbf{u}$ with $S_1(\mathbf{u}) \equiv 1 \bmod 2$. Therefore the sum becomes
\begin{align}
\label{ekMoment}
\mathcal{S}(X, k, \bm{a}) := \sum_{(D_\mathbf{u}) \in \mathcal{D}(X, k, \bm{a})} \prod_\mathbf{u} \frac{F(D_\mathbf{u})}{2^{k\omega(D_\mathbf{u})}} \prod_{\mathbf{u}, \mathbf{v}} \left(\frac{D_\mathbf{u}}{D_\mathbf{v}}\right)^{\phi(\mathbf{u}, \mathbf{v})},
\end{align}
where $\mathcal{D}(X, k, \bm{a})$ is the set of tuples of odd, square-free, positive and coprime integers $D_\mathbf{u}$, indexed by those $\mathbf{u} \in \mathbb{F}_2^{2k}$ with $S_1(\mathbf{u}) \equiv 0 \bmod 2$, satisfying
\begin{align}
\label{eClassSummation}
\prod_{\mathbf{u} \in \mathbb{F}_2^{2k}} D_\mathbf{u} \leq X, \quad D_\mathbf{u} \equiv a_\mathbf{u} \bmod 4.
\end{align}

\begin{definition}
\label{dUnlinked}
Let $\mathbf{u}, \mathbf{v} \in \mathbb{F}_2^{2k}$. We call $\mathbf{u}, \mathbf{v}$ unlinked if $\phi(\mathbf{u}, \mathbf{v}) + \phi(\mathbf{v}, \mathbf{u}) = 0$. A set $\mathcal{U} \subseteq \mathbb{F}_2^{2k}$ is called unlinked if $\phi(\mathbf{u}, \mathbf{v}) = 0$ for all $\mathbf{u}, \mathbf{v} \in \mathcal{U}$, and it is called maximally unlinked if it is a maximal unlinked set with respect to inclusion of sets.
\end{definition}

\noindent The point of this definition is that it records the presence of the Legendre symbol $(D_\mathbf{u}/D_\mathbf{v})$, where we make sure that the flipped term $(D_\mathbf{v}/D_\mathbf{u})$ does not occur in the product in equation \eqref{ekMoment}. Thus we expect oscillation coming from this Legendre symbol.

\begin{lemma}
\label{lUnlinked}
Let $\mathcal{U}$ be an unlinked set. Then we have $|\mathcal{U}| \leq 2^k$.
\end{lemma}

\begin{proof}
We define $P(\mathbf{w}) = \sum_{j = 0}^{k - 1} w_{2j + 1} (w_{2j + 1} + w_{2j + 2})$. With this definition set, we check that
\[
P(\mathbf{u} + \mathbf{v}) = \phi(\mathbf{u}, \mathbf{v}) + \phi(\mathbf{v}, \mathbf{u}).
\]
Then our lemma is a consequence of \cite[Lemma 18]{FK4}.
\end{proof}

\subsubsection{Higher Selmer moments}
We shall be brief as the manipulations are direct analogues of those in \S\ref{ssHigherClass}. In this case $\mathbb{F}_2^4$ will play the role of $\mathbb{F}_2^2$. We write $\pi_1, \dots, \pi_k$ for the projection map of $\mathbb{F}_2^{4k} \cong (\mathbb{F}_2^4)^k$ on the $i$-th copy of $\mathbb{F}_2^4$. We introduce the notations
\begin{gather*}
\varphi_i(\mathbf{u}, \mathbf{v}) := \psi(\pi_i(\mathbf{u}), \pi_i(\mathbf{v})) \\
\varphi(\mathbf{u}, \mathbf{v}) := \sum_{i = 1}^k \varphi_i(\mathbf{u}, \mathbf{v}),
\end{gather*}
and we let $S_1(\mathbf{u})$ be the number of $1 \leq i \leq k$ such that 
$$
\pi_i(\mathbf{u}) \in \{(1, 0, 1, 1), (1, 0, 0, 1), (1, 1, 1, 0), (0, 1, 1, 0), (1, 1, 0, 1), (0, 1, 1, 1)\}. 
$$
Then, after fixing congruence classes $\bm{a}$, it suffices to bound
$$
\mathcal{S}(X, k, \bm{a}) := \sum_{(D_\mathbf{u}) \in \mathcal{D}(X, k, \bm{a})} \prod_\mathbf{u} \frac{F(D_\mathbf{u})}{4^{k \omega(D_\mathbf{u})}} \prod_{\mathbf{u}, \mathbf{v} \in \mathbb{F}_2^{4k}} \left(\frac{D_\mathbf{u}}{D_\mathbf{v}}\right)^{\phi(\mathbf{u}, \mathbf{v})},
$$
where $\mathcal{D}(X, k, \bm{a})$ is the set of tuples of square-free, positive and coprime integers $D_\mathbf{u}$, indexed by those $\mathbf{u} \in \mathbb{F}_2^{4k}$ with $S_1(\mathbf{u}) \equiv 0 \bmod 2$, satisfying
\begin{align}
\label{eSelmerSummation}
\prod_{\mathbf{u} \in \mathbb{F}_2^{4k}} D_\mathbf{u} \leq X, \quad D_\mathbf{u} \equiv a_\mathbf{u} \bmod{8\Omega}, \quad \gcd(D_\mathbf{u}, 8\Omega) = 1.
\end{align}
Here we fixed an invertible congruence class $\bm{a} = (a_\mathbf{u})_{\mathbf{u} \in \F_2^{4k}}$ modulo $8 \Omega$ for each $D_\mathbf{u}$, which guarantees that $\lambda((D_\mathbf{u})_{\mathbf{u} \in \mathbb{F}_2^{4k}})$ is constant on $\mathcal{D}(X, k, \bm{a})$. The analogues of 
Definition~\ref{dUnlinked} and Lemma \ref{lUnlinked2} are:
\begin{definition}
Let $\mathbf{u}, \mathbf{v} \in \mathbb{F}_2^{4k}$. We call $\mathbf{u}, \mathbf{v}$ unlinked if $\phi(\mathbf{u}, \mathbf{v}) + \phi(\mathbf{v}, \mathbf{u}) = 0$. A set $\mathcal{U} \subseteq \mathbb{F}_2^{4k}$ is called unlinked if $\phi(\mathbf{u}, \mathbf{v}) = 0$ for all $\mathbf{u}, \mathbf{v} \in \mathcal{U}$, and it is called maximally unlinked if it is a maximal unlinked set with respect to inclusion of sets.
\end{definition}

\begin{lemma}
\label{lUnlinked2}
Let $\mathcal{U}$ be an unlinked set. Then we have $|\mathcal{U}| \leq 4^k$.
\end{lemma}

\begin{proof}
This is \cite[Lemma 7]{HBCongruent}.
\end{proof}

\subsection{Bounds for character sums}
Recall Definition \ref{dSuitableh}. Since $F$ is treated as fixed for us, we make once and for all a valid choice of $L$, $\alpha$ and $C(\alpha)$ as in Definition \ref{dSuitableh}, and allow all our implied constants to implicitly depend on the aforementioned choices.

\begin{terminology}
Let $A_1 > 0$ be a sufficiently small real number and $A_2 > 0$ be a sufficiently large real number, both to be chosen later in terms of $k$ only. 
We say that an integer $m$ is 
\begin{itemize}
\item large if $m > \exp\left((\log X)^{A_1}\right)$,
\item medium if $m> (\log X)^{A_2}$,
\item active if $m> 1$, $m \not \in E$ and $4m \not \in E$.
\end{itemize}
\end{terminology}

To allow for a uniform notation between the class group and Selmer group cases, we set 
 $M := \mathbb{F}_2^2$ in the former case and $M := \mathbb{F}_2^4$ in the latter. 
We let $b$ stand for the dimension of $M$ and set $$
\mathcal{S}(X, k, \bm{a}) := \sum_{(D_\mathbf{u}) \in \mathcal{D}(X, k, \bm{a})} \prod_\mathbf{u} \frac{F(D_\mathbf{u})}{b^{k \omega(D_\mathbf{u})}} \prod_{\mathbf{u}, \mathbf{v} \in M^k} \left(\frac{D_\mathbf{u}}{D_\mathbf{v}}\right)^{\phi(\mathbf{u}, \mathbf{v})},
$$
where $\mathbf{u}$ runs over all indices in $M^k$ with $S_1(\mathbf{u}) \equiv 0 \bmod 2$, and where we impose the summation conditions \eqref{eClassSummation} in the class group case and \eqref{eSelmerSummation} in the Selmer group case.

At this stage we partition $\mathcal{S}(X, k, \bm{a})$ into various pieces according to the sizes of the variables. As a first step, we define $\mathcal{S}_{\text{sm}}(X, k, \bm{a})$ to be the contribution to $\mathcal{S}_{\text{sm}}(X, k, \bm{a})$ for which there exist at most $b^k - 1$ large variables $D_\mathbf{u}$. The next lemma disposes of the contribution from $\mathcal{S}_{\text{sm}}(X, k, \bm{a})$ by showing that it is negligible.

\begin{lemma}
\label{lsm}
There exists some constant $c > 0$, depending only on $k,L$ and $\alpha$, such that
\[
\mathcal{S}_{\textup{sm}}(X, k, \bm{a}) \ll_k \frac{X}{(\log X)^{1+c}} \prod_{p \leq X} \left(1 + \frac{F(p)}{p}\right).
\]
\end{lemma}

\begin{proof}
Let $\mathcal{L}$ be any subset of $M^k$ of cardinality $|\mathcal{L}| = r \leq b^k - 1$. Since the number of choices for $\mathcal{L}$ is bounded in terms of $k$ only, it suffices to bound the contribution to $\mathcal{S}_{\text{sm}}(X, k, \bm{a})$, where we demand that $D_\mathbf{u}$ is large if and only if $\mathbf{u} \in \mathcal{L}$. We write $n$ for the product of those $D_\mathbf{u}$ with $\mathbf{u} \in \mathcal{L}$, and we write $m$ for the product of the remaining $D_\mathbf{u}$. Therefore we obtain the bound
\[
\mathcal{S}_{\textup{sm}}(X, k, \bm{a}) \ll_k \sum_{m \leq \exp\left((b^{2k} - r)(\log X)^{A_1}\right)} 
\frac{\mu(m)^2 F(m) \tau_{b^{2k} - r}(m)}{ b^{k\omega(m)} }
\sum_{n \leq X/m} \frac{\mu(n)^2 F(n) \tau_r(n)}{b^{k\omega(n)}}.
\]
The inner sum may be bounded by \cite[Corollary 2.15]{MV}. Feeding this in, we get
$$ 
\ll_k \frac{X}{\log X} \prod_{p \leq X} \left(1 + \frac{r b^{-k} F(p)}{p}\right) \sum_{m \leq \exp\left((b^{2k} - r)(\log X)^{A_1}\right)} 
\frac{\mu(m)^2 F(m) \tau_{b^{2k} - r}(m) }{ mb^{k\omega(m)}}.
$$
Bounding the harmonic sum by the corresponding Euler product yields the estimate
$$
\ll_k \frac{X}{\log X} \prod_{p \leq X} \left(1 + \frac{ F(p)}{p}\right)^{r b^{-k}}
\prod_{p \leq \exp\left((b^{2k} - r)(\log X)^{A_1}\right)} \left(1 + \frac{(b^{2k} - r) b^{-k} F(p)}{p}\right).
$$
Setting $\zeta=A_12^L(b^k-rb^{-k})$ we use the assumption $F(p)\leq 2^L$ to see that the second product is $\ll (\log X)^\zeta$. Let us introduce the strictly positive constant $\epsilon':=1-rb^{-k}$. We get 
$$
\ll_k 
\bigg\{ 
(\log X)^{\zeta}
\prod_{p \leq X} \left(1 + \frac{ F(p)}{p}\right)^{-\epsilon}
\bigg\} \frac{X}{\log X} 
\prod_{p \leq X} \left(1 + \frac{ F(p)}{p}\right)
$$ 
and note that the quantity inside the brackets $\{\}$ is $\ll (\log X)^{\zeta-\epsilon \alpha}$ by \eqref{eq:Calpha}. Upon taking $A_1$ sufficiently small in terms of $\alpha,k$ and $L$ ensures that $\zeta-\epsilon \alpha<0$, thus concluding the proof.
\end{proof}

Denote the contribution to $\mathcal{S}(X, k, \bm{a})$ for which there exist linked indices $\mathbf{u}, \mathbf{v}$ such that $D_\mathbf{u}$ and $D_\mathbf{v}$ are medium as $\mathcal{S}_{\text{LS}}(X, k, \bm{a})$. Similarly, we let $\mathcal{S}_{\text{SW}}(X, k, \bm{a})$ be the contribution for which 
\begin{itemize}
\item if $\mathbf{u}$, $\mathbf{v}$ are linked, then $D_\mathbf{u}$ or $D_\mathbf{v}$ is not medium,
\item there exist linked indices $\mathbf{u}, \mathbf{v}$ such that $D_\mathbf{u}$ is large and $D_\mathbf{v}$ is active.
\end{itemize}

\begin{lemma}
\label{lLS}
We have $\mathcal{S}_{\textup{LS}}(X, k, \bm{a}) \ll_k X(\log X)^{-100}$.
\end{lemma}

\begin{proof}
By the union bound, we may fix two linked indices $\mathbf{u}$ and $\mathbf{v}$ such that $D_\mathbf{u}$ is large and $D_\mathbf{v}$ is medium. 
This can be dealt with directly from Corollary \ref{cLargeSieve} with $z=(\log X)^{A_2}$, $s=L$ and $r=b^{2k}$.
This gives the stated bound upon choosing $A_2$ sufficiently large in terms of $b,k$ and $L$.
\end{proof}

\begin{lemma}
\label{lSW}
We have
 $
\mathcal{S}_{\textup{SW}}(X, k, \bm{a}) \ll_k X(\log X)^{-100}
 $.
\end{lemma}

\begin{proof}
By the union bound, we may fix two linked indices $\mathbf{u}$ and $\mathbf{v}$ such that $D_\mathbf{u}$ is large and $D_\mathbf{v}$ is active. Furthermore, if $\mathbf{u}$ and $\mathbf{v}$ are linked, then $D_\mathbf{v}$ is not medium. We now isolate the variable $D_\mathbf{u}$ by applying the triangle inequality. We apply Theorem \ref{tKou} to the resulting inner sum. To check that this application of Theorem \ref{tKou} is permitted, we need to verify that \eqref{eOnPrimes} holds. We claim that this follows from assumption \eqref{eSW} (for a large choice of $A$ in terms of $k$ and $L$) and the definition of active. 

Indeed, the character $(\cdot/D_\mathbf{v})$ has conductor $D_\mathbf{v}$, so this character is not in $E$ by definition of active. The symbol $(D_\mathbf{v}/\cdot)$ is not a Dirichlet character, but when restricted to odd positive arguments, it is equal to a Dirichlet character of conductor $4D_\mathbf{v}$, which is also not in $E$ by definition of active. The resulting Dirichlet characters are also readily verified to be non-principal, quadratic and primitive for any odd, square-free integer $D_\mathbf{v} > 1$. Note that the total conductor is indeed bounded by a power of $\log X$, since all variables $D_\mathbf{v}$ with $\mathbf{v}$ linked to $\mathbf{u}$ are not medium. We take $\epsilon = 1/2$, $Q := \exp((\log X)^{A'})$ for some very small $A' > 0$ in terms of $k$ and $J$ sufficiently large in terms of $k$.

Summing trivially over all the other variables as in the proof of the previous lemma gives the stated bound.
\end{proof}

\begin{theorem}
\label{tMoments}
Let $k \in \Z_{\geq 1}$, $\bm{a} = (a_\mathbf{u})_{\mathbf{u} \in M^k}$ and assume $F$ is appropri\'ee.
Then 
$$
\mathcal{S}(X, k, \bm{a}) \ll_k \frac{X}{\log X} \prod_{p \leq X} \left(1 + \frac{F(p)}{p}\right).
$$
\end{theorem}

\begin{proof}
We split $\mathcal{S}(X, k, \bm{a})$ in $\ll_k 1$ subsums depending on the sizes of the variables $D_\mathbf{u}$. Write $\mathcal{L}$ for the set of indices $\mathbf{u}$ for which $D_\mathbf{u}$ is large and write $\mathcal{M}$ for the set of indices for which $D_\mathbf{u}$ is medium, so $\mathcal{L} \subseteq \mathcal{M}$. If $|\mathcal{L}| \leq b^k - 1$, then the resulting subsums fall under the purview of $\mathcal{S}_{\text{sm}}(X, k, \bm{a})$, and thus we appeal to Lemma \ref{lsm} to bound their contribution to $\mathcal{S}(X, k, \bm{a})$. 

It remains to bound the cases where $|\mathcal{L}| \geq b^k$. If there exist linked indices $\mathbf{u} \in \mathcal{M}$ and $\mathbf{v} \in \mathcal{M}$, we may appeal to Lemma \ref{lLS} to show that the resulting contribution is in $\mathcal{S}_{\text{LS}}(X, k, \bm{a})$ and therefore negligible. 
In the remaining cases 
all elements $\mathbf{u}, \mathbf{v} \in \mathcal{L}$ are unlinked. Hence Lemma \ref{lUnlinked} and Lemma \ref{lUnlinked2} force that $\mathcal{L}$ is maximally unlinked, and thus $|\mathcal{L}| = b^k$.

In the remaining subsums, we must have $|\mathcal{M}| = b^k$. Indeed, $\mathcal{L}$ is maximally unlinked, so for every $\mathbf{u} \in \mathcal{M}$, there exists $\mathbf{v} \in \mathcal{L}$ such that $\mathbf{u}$ and $\mathbf{v}$ are linked. Therefore such subsums fall under the purview of $\mathcal{S}_{\text{LS}}(X, k, \bm{a})$, which we have already shown to be negligible. Now define $\mathcal{A}$ to be the set of $\mathbf{u} \in \mathcal{A}$ such that $D_\mathbf{u}$ is active. If $|\mathcal{A}| > |\mathcal{L}|$, then the resulting contribution to $\mathcal{S}(X, k, \bm{a})$ is negligible due to Lemma \ref{lSW}.

At this stage, the only remaining subsums satisfy $|\mathcal{L}| = |\mathcal{M}| = |\mathcal{A}| = b^k$. Therefore we see that $D_\mathbf{u} = 1$, $D_\mathbf{u} \in E$ or $4D_\mathbf{u} \in E$ for all $\mathbf{u} \not \in \mathcal{L}$. Since there are only finitely many exceptional moduli in the set $E$, we first fix the variables outside $\mathcal{L}$, then trivially bound each quadratic symbol by $1$. Let $t:=b^k$ denote the number of large variables. Then the resulting sum will be 
$$ 
\ll \sum_{b_1\cdots b_t \leq X} \frac{\mu(b_1\cdots b_k)^2 F(b_1\cdots b_k)}{t^{\omega(b_1\cdots b_k)}} = \sum_{b\leq X} \mu(b)^2F(b).
$$ 
Alluding to Shiu's bound \cite[Theorem 1]{Shiu} concludes the proof.
\end{proof}

We are now ready to prove the main result of this section.

\begin{proof}[Proof of Theorem \ref{tCharacter}]
The result is a direct consequence of Theorem \ref{tMoments}, since
$$
\sum_{1 \leq m \leq X} \mu(2m)^2 f_k^\ast(m) F(m) \leq \sum_{\bm{a}} \left|\mathcal{S}(X, k, \bm{a})\right|,
$$
and there are at most $\ll_k 1$ choices of $\bm{a}$.
\end{proof}

\section{Proof of main theorems}
\label{sFinal}
\subsection{Proof of Theorem \ref{t12}}
\label{ss:prf1}
We are now ready to prove Theorem \ref{t12}. The overarching logic is that Theorem \ref{t4rankSub} and Lemma \ref{lh3Level} will allow us to employ Theorem \ref{tNT}. More precisely, Theorem \ref{t4rankSub} gives that the moments of the $4$-rank have a R\'edei majorant and Lemma \ref{lh3Level} gives the required level of distribution result for $h_3(n)$. The sieving process of Theorem \ref{tNT} will produce a linear sum over all integers containing the twisted $4$-rank $g(m, n)$ from Subsection \ref{ss4rank} weighted by the density function $\delta(m)$ of $h_3(n)$ introduced in \eqref{eh3Density}. This final sum is handled by an appeal to Theorem \ref{tCharacter}.

\begin{proof}[Proof of Theorem \ref{t12}]
Let $k \geq 1$ and $n = 3 \cdot 2^k$. Since we have $h_n(d) \geq h_2(d) \geq 2^{\omega(d) - 2}$, the lower bound is trivial. For the upper bound, we will prove that
$$
\sum_{\substack{0 < d \leq X \\ \textup{ fundamental}}} h_{3 \cdot 2^k}(d) \ll_k X \log X.
$$
The negative discriminants can be dealt with in a similar fashion.

Since $h_{2^{t + 1}}(d)/h_{2^t}(d) \leq h_{2^t}(d)/h_{2^{t - 1}}(d)$ for $t \geq 1$, we deduce that 
\[
h_{2^k}(d) = h_2(d) \frac{h_4(d)}{h_2(d)} \frac{h_8(d)}{h_4(d)} \cdots \frac{h_{2^k}(d)}{h_{2^{k-1}}(d)} \leq h_2(d) \bigg(\frac{h_4(d)}{h_2(d)} \bigg)^{k-1} = h_2(d) 2^{(k-1) \cdot\mathrm{rk}_4 \mathrm{Cl}(\Q(\sqrt{d}))}.
\] Using $h_2(d)\leq 2^{\omega(d)}$ we see that 
$h_{3 \cdot 2^k}(d) \leq h_3(d) 2^{\omega(d)} 2^{k \cdot\mathrm{rk}_4 \mathrm{Cl}(\Q(\sqrt{d}))}$.
Therefore, \[\sum_{\substack{0 < d \leq X \\ \textup{ fundamental}}} h_{3 \cdot 2^k}(d) \leq \sum_{\substack{0 < d \leq X \\ \textup{ fundamental}}} 
2^{\omega(d)} 2^{k \cdot \mathrm{rk}_4 \mathrm{Cl}(\Q(\sqrt{d}))} + \sum_{\substack{0 < d \leq X \\ \textup{ fundamental}}} (h_3(d) - 1) 
 2^{\omega(d)} 2^{k \cdot \mathrm{rk}_4 \mathrm{Cl}(\Q(\sqrt{d}))}. \]
The special case $\kappa = 2$ of the work of Fouvry--Kl\"uners~\cite[Equation (53)]{FKWeighted} shows that the first sum 
in the right-hand side is $\ll_k X \log X$. Therefore, it suffices to show that
$$\sum_{\substack{0 < d \leq X \\ \textup{ fundamental}}} (h_3(d) - 1) 2^{\omega(d)} 2^{k \cdot \mathrm{rk}_4 \mathrm{Cl}(\Q(\sqrt{d}))} \ll_k X \log X.$$
At this point we apply Theorem \ref{tNT} with $f(d) = 2^{\omega(d)} 2^{k \cdot \mathrm{rk}_4 \mathrm{Cl}(\Q(\sqrt{d}))}$, $a_d = d$ and 
weights given by $w_d = \mathds{1}_{d \text{ fundamental}} \times (h_3(d) - 1)$. For each odd prime $p$ and $e \in \Z_{\geq 1}$, we take the partition $\mathcal{P}(p^e)$ to be $\{\mathcal{A}_1, \mathcal{A}_2, \mathcal{A}_3\}$, where $\mathcal{A}_1$ consists of the invertible squares inside $\Z/p^e\Z$, $\mathcal{A}_2$ consists of the invertible non-squares in $\Z/p^e\Z$ and $\mathcal{A}_3$ consists of all elements divisible by $p$, while for $p = 2$ we partition into the odd and even numbers. The function $g$ is the one from Definition \ref{def:firstdefn}. To see why $f$ is $(A,2^\omega\cdot g)$-R\'edei majorized we use Theorem \ref{t4rankSub} to get the inequality
$$
2^{k \cdot \mathrm{rk}_4 \mathrm{Cl}(\Q(\sqrt{m n}))} \leq g(m, n)^k \cdot 2^{k \omega(n) + k}.
$$
The majorization then follows from the inequality
$$
2^{\omega(mn)} \cdot 2^{k \cdot \mathrm{rk}_4 \mathrm{Cl}(\Q(\sqrt{m n}))} \leq 2^{\omega(m)} \cdot g(m, n)^k \cdot 2^{(k + 1) \omega(n) + k}.
$$ 
The sequence $h_3(d) - 1$ has a positive level of distribution thanks to Lemma \ref{lh3Level} with the choice $M(X)=X/\pi^2$. Recall the density function $\delta(m)$ defined in \eqref{eh3Density}; the function $h(d,\c A)$ is defined by the level of distribution result in Lemma \ref{lh3Level}. One readily checks that $\delta(m)$ satisfies the hypotheses of Theorem \ref{tNT}, where the hypothesis \eqref{eHSubm} follows by adapting the proof of Theorem \ref{t4rankSub}. This motivates us to introduce the quantity
$$
f^\ast(m) = \frac{2^{\omega(m)}}{2^r} \sum_{\bm{\epsilon}} g(m, \bm{\epsilon})^k,
$$
where $m$ has exactly $r$ odd prime divisors. Then Theorem \ref{tNT} yields
\begin{align*}
\sum_{\substack{0 < d \leq X \\ \textup{ fundamental}}} (h_3(d) - 1) \cdot 2^{\omega(d)} \cdot 2^{k \cdot \mathrm{rk}_4 \mathrm{Cl}(\Q(\sqrt{d}))} 
&\ll_k X \prod_{p \leq X} (1 - \delta(p)) \sum_{a \leq X} f^\ast(a) \delta(a) \\
&\ll \frac{X}{\log X} \sum_{a \leq 8X} \frac{f^\ast(a) \mu(2a)^2}{a},
\end{align*}
since $\delta(a) \ll \frac{1}{a}$. The last inequality uses that $f^\ast(a)$ depends only on the largest odd square-free divisor of $a$ and that $\delta(a)$ vanishes if $16 \mid a$ or $p^2 \mid a$ for $p \geq 3$. After applying partial summation, it suffices to show that
\[
\sum_{a \leq t} f^\ast(a) \mu(2a)^2 \ll_k t \log t.
\]
Taking $F$ to be the multiplicative function $2^{\omega(a)}$, this follows from Theorem \ref{tCharacter}.
\end{proof}

\subsection{Proof of Theorem \ref{tFibration}}
\label{ss:prf2}
The overall logic will be similar to the proof of Theorem \ref{t12}. In this case Theorem \ref{tSelmerSub} and Lemma \ref{lPolyLevel} will play the role of Theorem \ref{t4rankSub} and Lemma \ref{lh3Level}. We then apply Theorem \ref{tNT}. The resulting linear sum is however not necessarily over square-free values. For this reason we first apply Lemma \ref{tSqfReduction} before we are able to use Theorem \ref{tCharacter}. 

\begin{proof}[Proof of Theorem \ref{tFibration}]
We start by remarking that the lower bound is trivial, so it suffices to establish the upper bound, for which we first make some reductions. Recall that our elliptic fibration $f: \mathcal{E} \rightarrow \mathbb{A}^n$ is given by $P(t_1, \dots, t_n) y^2 = (x - r_1) (x - r_2) (x - r_3)$. By removing square factors from the polynomial $P$, we may reduce to the case that $P$ is separable. Furthermore, if $P$ is a non-zero constant, then the upper bound is trivial. Henceforth we will assume that $P$ has degree at least $1$. Furthermore, we may reduce to the case where $\gcd(r_1, r_2, r_3) = 1$ by quadratic twisting our elliptic curve if necessary.

Hence it is enough to establish that for all separable non-constant polynomials $P \in \Z[t_1, \dots, t_n]$ and all $\kappa > 1$ there exists $C > 0$ such that for all $B \geq 3$ one has 
$$
\sum_{\substack{\b t\in \Z^n, P(\b t)\neq 0 \\ \max_i |t_i| \leq B}} \kappa^{\mathrm{rk}(E(\mathbf{t}))} \leq CB^n.
$$
We let $k \geq 1$ be the smallest integer such that $\kappa \leq 2^k$, and recall that $\Omega := 2(r_1 - r_2)(r_1 - r_3)(r_2 - r_3)$. By Theorem \ref{tSelmerSub} there exists a finite collection $\mathcal{C}$ such that
$$ 
\sum_{\substack{\b t\in \Z^n, P(\b t)\neq 0 \\ \max_i |t_i| \leq B}} \kappa^{\mathrm{rk}(E(\mathbf{t}))} \leq \sum_{\substack{\b t\in \Z^n, P(\b t)\neq 0 \\ \max_i |t_i| \leq B}} |\mathrm{Sel}^2(E(\mathbf{t}))|^k \leq 4^{k \cdot |\Omega| + k} \max_{\mathbf{r} \in \mathcal{C}}\sum_{\substack{\b t\in \Z^n, P(\b t) \neq 0 \\ \max_i |t_i| \leq B}} f_\mathbf{r}(P(t_1, \dots, t_n))^k,
$$
where $f_\mathbf{r}$ is introduced in \S\ref{ssSelmer}. We fix some $\mathbf{r} \in \mathcal{C}$ and we aim to upper bound each individual sum
$$
\sum_{\substack{\b t\in \Z^n, P(\b t)\neq 0 \\ \max_i |t_i| \leq B}} f_\mathbf{r}(P(t_1, \dots, t_n))^k.
$$
We estimate this sum with Theorem \ref{tNT} by first parametrising the elements $\b t \in \Z^n$ through the integers $n \in \Z_{\geq 1}$. Because $P$ is separable and has degree at least $1$, we may apply Lemma \ref{lPolyLevel}, and we write $h(m)$ and $h(m, \bm{\epsilon})$ for the resulting density functions. Note that the condition $P(\mathbf{t}) \neq 0$ may be ignored as this set can be shown to be of size $O(B^{n - 1})$. Because of the second part of Lemma \ref{lPolyLevel}, these density functions satisfy the required conditions to apply Theorem \ref{tNT}, where \eqref{eHSubm} follows by adapting the proof of Theorem \ref{tSelmerSub}. Furthermore, we have that
$$
f_\mathbf{r}(mn)^k \leq g_\mathbf{r}(m, n)^k 4^{k \cdot \omega(n)}
$$
by Theorem \ref{tSelmerSub}, so $f_\mathbf{r}$ is R\'edei majorized. Thus, Theorem \ref{tNT} provides us with the upper bound
\begin{align}
\label{eNToutput}
\frac{1}{B^n} 
\sum_{\substack{\mathbf{t} \in \Z^n, P(\mathbf{t}) \neq 0 \\ \max_i |t_i| \leq B }} f_\mathbf{r}(P(\b t))^k 
\ll_{k, P} \prod_{p \leq B^n} (1 - h(p)) \sum_{1 \leq a \leq B^n} f^\ast(a),
\end{align}
where
$$
f^\ast(a) = \sum_{\bm{\epsilon} = (\epsilon_i)_{1 \leq i \leq \omega_1(a)}} g(a, \bm{\epsilon}) h(a, \bm{\epsilon}) = \frac{1}{2^{\omega_1(a)}} \sum_{\bm{\epsilon} = (\epsilon_i)_{1 \leq i \leq \omega_1(a)}} g(a, \bm{\epsilon}) 2^{\omega_1(a)} h(a, \bm{\epsilon})
$$
with $\omega_1(a)$ the number of prime divisors of $a$ coprime to $\Omega$. By Lemma \ref{lPolyLevel} there exists a constant $C_6 > 0$ and a multiplicative function $\tilde{h}$ such that
\begin{align}
\label{ehconstruct}
2^{\omega_1(a)} h(a, \bm{\epsilon}) \leq h(a) \tilde{h}(a), \quad \quad 1 - \frac{C_6}{p^2} \leq \tilde{h}(p) \leq 1 + \frac{C_6}{p^2}, \quad \quad \tilde{h}(p^e) \leq 2 \text{ for all } e \geq 1.
\end{align}
We will now bound
$$
\sum_{1 \leq a \leq B^n} f^\ast(a) \leq \sum_{1 \leq a \leq B^n} \frac{h(a) \tilde{h}(a)}{2^{\omega_1(a)}} \sum_{\bm{\epsilon} = (\epsilon_i)_{1 \leq i \leq \omega_1(a)}} g(a, \bm{\epsilon}) = \sum_{1 \leq a \leq B^n} h(a) \tilde{h}(a) \tilde{f}^\ast(a),
$$
where $\tilde{f}^\ast(a) = 2^{-\omega_1(a)} \sum_{\bm{\epsilon}} g(a, \bm{\epsilon})$. One directly checks that $\tilde{f}^\ast(a)$ satisfies the conditions of Lemma \ref{tSqfReduction}, while for $h(a) \tilde{h}(a)$ this is a consequence of Lemma \ref{lPolyLevel} and \eqref{ehconstruct}. It suffices to show that
\begin{align}
\label{eSquarefreeReduction}
\sum_{1 \leq a \leq B^n} \mu(a)^2 h(a) \tilde{h}(a) \tilde{f}^\ast(a) \ll \prod_{p \leq B^n} \left(1 + h(p) \tilde{h}(p)\right).
\end{align}
Indeed, if so, we apply Lemma \ref{tSqfReduction} to deduce that
$$
\sum_{1 \leq a \leq B^n} f^\ast(a) \leq \sum_{1 \leq a \leq B^n} h(a) \tilde{h}(a) \tilde{f}^\ast(a) \ll \prod_{p \leq B^n} \left(1 + h(p) \tilde{h}(p)\right).
$$
The theorem is proved by injecting the above bound into \eqref{eNToutput} and using the simple estimate $\prod_{p \leq B^n} (1 - h(p)) (1 + h(p) \tilde{h}(p)) \ll 1$.

In order to establish the claim \eqref{eSquarefreeReduction}, we define the new multiplicative function $\bar{h}(a) = a \cdot h(a) \cdot \tilde{h}(a)$. By partial summation it is enough to demonstrate the inequality
$$
\sum_{1 \leq a \leq t} \mu(a)^2 \bar{h}(a) \tilde{f}^\ast(a) \ll_{k, P} \frac{t}{\log t} \prod_{p \leq t} \left(1 + \frac{\bar{h}(p)}{p}\right).
$$
To finish the proof, it remains to verify that $\bar{h}$ satisfies the conditions $(i)$, $(ii)$ and $(iii)$ in Theorem \ref{tCharacter}. The Lang--Weil bounds show that
\[
h(p) = \frac{c_P(p)}{p} + O\left(p^{-3/2}\right),
\]
where $c_P(p)$ is the number of distinct irreducible factors of $P$ defined over $\mathbb{F}_p$. The map $p \mapsto c_P(p)$ is Frobenian, i.e. is determined by the splitting of $p$ in a fixed number field. Furthermore, the average of $c_P(p)$ over the primes is equal to the number of distinct irreducible factors of $P$ over $\mathbb{Q}$. Therefore the conditions $(i)$, $(ii)$, $(iii)$ readily follow from \cite[Lemma 2.5]{LM}.
\end{proof}


\begin{thebibliography}{19}
\bibitem{BF}
K. Belabas and \'E. Fouvry.
Discriminants cubiques et progressions arithm\'etiques.
\textit{Int. J. Number Theory} \textbf{6}(7) (2010), 1491--1529.

\bibitem{MR3090184}
M. Bhargava, A. Shankar and J. Tsimerman. 
On the {D}avenport-{H}eilbronn theorems and second order terms. 
\textit{Invent. Math.} \textbf{193}(2) (2013), 439--499.
 
\bibitem{BTT}
M, Bhargava, T. Taniguchi and F. Thorne. 
Improved error estimates for the Davenport--Heilbronn theorems. 
 \textit{Math. Ann.} \textbf{389} (2024), 3471--3512.

\bibitem{CKPS} 
S. Chan, P. Koymans, C. Pagano and E. Sofos.
 Averages of multiplicative functions along equidistributed sequences.
\textit{arXiv preprint:}2402.08710.

\bibitem{CKPS2} 
\bysame.
6-torsion and integral points on quartic threefolds.
\textit{arXiv preprint:}2403.13359.

\bibitem{CTSSD}
J.-L. Colliot-Th\'el\`ene, A.N. Skorobogatov and P. Swinnerton-Dyer.
Hasse principle for pencils of curves of genus one whose Jacobians have rational 2-division points.
\textit{Invent. Math.} \textbf{134} (1998), 579--650.

\bibitem{DH}
H. Davenport and H. Heilbronn. 
On the density of discriminants of cubic fields. II. 
\textit{Proc. Roy. Soc. London Ser. A} \textbf{322} (1971), 405--420.

\bibitem{Erdos}
P. Erd\H{o}s. 
On the sum $\sum_{k = 1}^x d(f(k))$.
\textit{J. London Math. Soc.} \textbf{27} (1952), 7--15.

\bibitem{FK} 
\'E. Fouvry and J. Kl\"uners.
Cohen--Lenstra heuristics of quadratic number fields.
\textit{Algorithmic number theory}, 40--55.
Lecture Notes in Comput. Sci., 4076, \emph{Springer-Verlag, Berlin}, 2006.

\bibitem{FK4} 
\bysame.
On the 4-rank of class groups of quadratic number fields. 
\textit{Invent. Math.} \textbf{167}(3) (2007), 455--513.

\bibitem{FKWeighted}
\bysame.
Weighted Distribution of the 4-rank of Class Groups and Applications.
\textit{International Mathematics Research Notices} \textbf{16} (2011), 3618--3656.

\bibitem{FP}
\'E. Fouvry and J. Pomykala.
Rang des courbes elliptiques et sommes d'exponentielles.
\textit{Monatsh. Math.} \textbf{116}(2) (1993), 111--126.

\bibitem{MR2675875}
F. Friedlander and H. Iwaniec.
Ternary quadratic forms with rational zeros.
\textit{J. Th\'{e}or. Nombres Bordeaux} \textbf{22}(1) (2010), 97--113.

\bibitem{GK}
A. Granville and D. Koukoulopoulos.
Beyond the LSD method for the partial sums of multiplicative functions.
\textit{Ramanujan J.} \textbf{49}(2) (2019), 287--319.

\bibitem{HBCongruent}
D.R. Heath-Brown.
The size of Selmer groups for the congruent number problem.
\textit{Invent. Math.} \textbf{111}(1) (1993), 171--196.

\bibitem{HB}
\bysame.
A mean value estimate for real character sums. 
\textit{Acta Arith.} \textbf{72}(3) (1995), 235--275. 

\bibitem{iwaniec}
H. Iwaniec, E. Kowalski.
Analytic number theory.
\textit{American Mathematical Society Colloquium Publications} 
\textbf{53} American Mathematical Society, Providence, RI, (2004), xii+615.
 
\bibitem{Kane}
D. Kane.
On the ranks of the 2-Selmer groups of twists of a given elliptic curve.
\textit{Algebra Number Theory} \textbf{7}(5) (2013), 1253--1279.

\bibitem{Kou}
D. Koukoulopoulos.
The distribution of prime numbers. 
Graduate Studies in Mathematics, 203. 
\textit{American Mathematical Society, Providence, RI,} 2019. 
xii + 356 pp. 

\bibitem{KR}
P. Koymans and N. Rome.
Weak approximation on the norm one torus.
\textit{Compos. Math.} \textbf{160}(6) (2024), 1304--1348.

\bibitem{LW}
S. Lang and A. Weil. 
Number of points of varieties in finite fields. 
\textit{Amer. J. Math.} \textbf{76} (1954), 819--827.

\bibitem{woodwangoliver}
R.J. Lemke Oliver, J. Wang and M.M. Wood.
The average size of $3$-torsion in class groups of $2$-extensions.
\textit{arXiv preprint:}2110.07712.

\bibitem{LM}
D. Loughran and L. Matthiesen.
Frobenian multiplicative functions and rational points in fibrations. 
\textit{J. Eur. Math. Soc.} \textbf{26} (2024), no. 12, 4779--4830.

\bibitem{Michel}
P. Michel.
Rang moyen de familles de courbes elliptiques et lois de Sato-Tate.
\textit{Monatsh. Math.} \textbf{120}(2) (1995), 127--136.

\bibitem{MV}
H.L. Montgomery and R.C. Vaughan.
Multiplicative number theory. I. Classical theory.
Cambridge Stud. Adv. Math., 97,
\textit{Cambridge University Press, Cambridge,} 2007. xviii+552 pp.

\bibitem{nair}
M. Nair.
Multiplicative functions of polynomial values in short intervals. 
\textit{Acta Arith.} \textbf{62} (1992), 257--269.

\bibitem{NT}
M. Nair and G. Tenenbaum.
Short sums of certain arithmetic functions. 
\textit{Acta Math.} \textbf{180} (1998), 119--144.

\bibitem{Neron}
A. N\'eron. 
Probl\`emes arithm\'etiques et g\'eome\'triques rattach\'es \`a la notion de rang d'une courbe alg\'ebrique dans un corps. 
\textit{Bull. Soc. Math. France} \textbf{80} (1952), 101--166.

\bibitem{1469.11173}
J. Park, B. Poonen, J. Voight and M.M. Wood.
A heuristic for boundedness of ranks of elliptic curves. 
\textit{J. Eur. Math. Soc.} \textbf{21} (2019), 2859--2903.

\bibitem{Redei}
L. R\'edei. 
Arithmetischer Beweis des Satzes \"uber die Anzahl der durch vier teilbaren Invarianten der absoluten Klassengruppe im quadratischen Zahlk\"orper. 
\textit{J. reine angew. Math.} \textbf{171} (1934), 55--60.

\bibitem{San}
T. Santens.
Diagonal quartic surfaces with a Brauer--Manin obstruction. 
\textit{Compos. Math.} \textbf{159} (4) (2023), 659--710.

\bibitem{Shiu}
P. Shiu.
A Brun--Titschmarsh theorem for multiplicative functions. 
\textit{J. reine angew. Math.} \textbf{313} (1980), 161--170.

\bibitem{Silverman1}
J. Silverman.
Heights and the specialization map for families of abelian varieties. 
\textit{J. reine angew. Math.} \textbf{342} (1983), 197--211.

\bibitem{Silverman2}
\bysame. 
Divisibility of the Specialization Map for Families of Elliptic Curves.
\textit{Amer. J. Math.} \textbf{107}(3) (1985), 555--565.

\bibitem{Smith1}
A. Smith.
The distribution of $\ell^\infty$-Selmer groups in degree $\ell$ twist families I.
\textit{arXiv preprint:}2207.05674.

\bibitem{Smith2}
\bysame. 
The distribution of $\ell^\infty$-Selmer groups in degree $\ell$ twist families II.
\textit{arXiv preprint:}2207.05143.
 
\bibitem{Stevenhagen} 
P. Stevenhagen.
R\'edei Matrices and Applications.
London Math. Soc. Lecture Note Ser., 215, 
\textit{Cambridge University Press, Cambridge,} 1995.
 
\bibitem{MR3127806}
T. Taniguchi and F. Thorne.
Secondary terms in counting functions for cubic fields. 
\textit{Duke Math. J.} \textbf{162} (2013), 2451--2508.

\bibitem{Watkins}
M. Watkins.
Distribution of the $2$-Selmer rank under twisting.
\textit{Publications math\'ematiques de Besan\ccc{c}on.} 
\textit{Alg\`ebre et th\'eorie des nombres} (2022), 59--133.

\bibitem{Wil}
C. Wilson.
General bilinear forms in the Jacobi symbol over hyperbolic regions. 
\textit{Monatsh. Math.} {\bf 202} (2023), 435--451.

\bibitem{Wolke}
D. Wolke.
Multiplikative Funktionen auf schnell wachsenden Folgen. 
\textit{J. reine angew. Math.} {\bf 251} (1971), 54--67.
\end{thebibliography}
\end{document}